\documentclass[a4paper,11pt,reqno]{article}

\usepackage[utf8]{inputenc}
\usepackage[T1]{fontenc}
\usepackage{lmodern}
\usepackage[english]{babel}
\usepackage{microtype}
\usepackage{tikz}

\usepackage{amsmath,amssymb,amsfonts,amsthm}
\usepackage{mathtools,accents}
\usepackage{mathrsfs}
\usepackage{aliascnt}
\usepackage{braket}
\usepackage{bm}
\usepackage[a4paper,margin=3cm]{geometry}
\usepackage[citecolor=blue,colorlinks]{hyperref}

\usepackage{enumerate}
\usepackage{xcolor}

\makeatletter
\g@addto@macro\@floatboxreset\centering
\makeatother

\makeatletter
\def\newaliasedtheorem#1[#2]#3{
  \newaliascnt{#1@alt}{#2}
  \newtheorem{#1}[#1@alt]{#3}
  \expandafter\newcommand\csname #1@altname\endcsname{#3}
}
\makeatother

\numberwithin{equation}{section}

\newtheoremstyle{slanted}{\topsep}{\topsep}{\slshape}{}{\bfseries}{.}{.5em}{}

\theoremstyle{plain}
\newtheorem{theorem}{Theorem}[section]
\newaliasedtheorem{proposition}[theorem]{Proposition}
\newaliasedtheorem{lemma}[theorem]{Lemma}
\newaliasedtheorem{corollary}[theorem]{Corollary}
\newaliasedtheorem{conjecture}[theorem]{Conjecture}
\newaliasedtheorem{counterexample}[theorem]{Counterexample}

\theoremstyle{definition}
\newaliasedtheorem{definition}[theorem]{Definition}
\newaliasedtheorem{question}[theorem]{Question}
\newaliasedtheorem{example}[theorem]{Example}

\theoremstyle{remark}
\newaliasedtheorem{remark}[theorem]{Remark}

\let\altphi\phi
\let\phi\varphi
\let\varphi\altphi
\let\altphi\undefined

\newcommand{\abs}[1]{\left\lvert#1\right\rvert}
\newcommand{\norm}[1]{\left\lVert#1\right\rVert}

\newcommand{\dd}{\mathop{}\!\mathrm{d}}

\DeclareMathOperator{\supp}{supp}

\newfont{\tmpf}{cmsy10 scaled 2500}

\newcommand{\intav}{{\mathop{\int\kern-10pt\rotatebox{0}{\textbf{--}}}}}

\renewcommand{\ }{\text{ }}
\def\<{\langle}
\def\>{\rangle}

\title{Fastest-Species Selection in Critical Multispecies Competition–Diffusion Systems via an Entropy Approach}
\author{
Haozhe Shu\thanks{Mathematical Institute, Tohoku University, Sendai, 980-8578, Japan; \url{shu.haozhe.t7@dc.tohoku.ac.jp }} \,and
Dongyuan Xiao\thanks{School of Mathematical Sciences, Shanghai Jiao Tong University, Shanghai, 200240, P.R. China; \url{dongyuanx@sjtu.edu.cn}}}
\date{\today}

\begin{document}
\maketitle

\begin{abstract}
We study an $N$-species competition--diffusion system with a common competition term, for which the associated ODE system admits infinitely many equilibria. For $N=2$, Alfaro and Xiao \cite{AlfaroXiao2023} studied the long-time behavior of solutions with compactly supported initial data by constructing precise super- and sub-solutions. For $N\geq 3$, the absence of a comparison principle for the entire system presents a major difficulty. Using a weighted entropy method, we prove that every bounded nonnegative traveling profile connecting nonzero equilibria is constant, without imposing any monotonicity assumption. For nontrivial compactly supported initial populations, we further show that the species with the strictly largest Fisher-KPP speed excludes all slower species and spreads at its scalar Fisher-KPP speed. Finally, we establish heat-type lower bounds for the slower species, revealing a bump phenomenon despite their uniform extinction.
\end{abstract}

\tableofcontents

\section{Introduction}
We consider the competition-diffusion system 
\begin{equation}\label{eq:competition_diffusion}
    \partial_t u_i=d_i\partial^2_xu_i+r_iu_i\left(1-k\sum_{j=1}^Nu_j^q\right),\quad t>0,\ x\in \mathbb{R},
\end{equation}
where $i=1,2,\ldots,N$, $k>0$, $N\geq 2$ and $q\geq 1$. Here $u_i$ denotes the density of the $i$-th species, while $d_i$ and $r_i$ are its diffusion coefficient and intrinsic growth rate, respectively. When $q=1$, this is the critical Lotka--Volterra competition system, in which the intra- and interspecific competition coefficients coincide. By the rescaling $\sqrt[q]{k}u_i\to u_i$, the parameter $k$ can be normalized to one. Thus, without loss of generality, we assume $k=1$ throughout this paper. For every $q\geq 1$, the $N$-species system has the continuum of nonzero equilibria 
\begin{equation}\label{eq:sec1_sigma_q}
    \Sigma_q:=\left\{a\in [0,\infty)^N:\sum_{i=1}^Na_i^q=1\right\}.
\end{equation}
In this paper, we address two questions. Whether nonconstant traveling waves can connect equilibria in $\Sigma_q$? When several species spread from a compact region, which of them survives, and how quickly do the others disappear? 

To compare the spreading abilities of the species, we introduce their Fisher--KPP spreading speeds. In the absence of competitors, each species satisfies an equation of the form 
\begin{equation}\label{eq:kpp}
    \partial_t u=d\partial^2_x u+ru(1-u^q),\quad t>0,\ x\in \mathbb{R},
\end{equation}
with $d,r>0$. When $q=1$, this is the classical equation introduced by Fisher \cite{Fisher1937} and Kolmogorov, Petrovskii and Piskunov \cite{KPP1937}. For every $q\geq 1$, (\ref{eq:kpp}) is of Fisher--KPP type where the reaction term $f(u)=ru(1-u^q)$ satisfies $f^\prime(0)=r$ and $0<f(u)\leq ru$ for $0<u<1$. A classical spreading result \cite{AronsonWeinberger1978} states that, the solution of (\ref{eq:kpp}) starting from nontrivial compactly supported initial data satisfies 
$$
\begin{aligned}
    &\lim_{t\to\infty}\sup_{\abs{x}\geq ct}\abs{u(t,x)}=0,\quad & \text{for all}\ c>2\sqrt{dr},\\
    &\lim_{t\to\infty}\sup_{\abs{x}\leq ct}\abs{1-u(t,x)}=0,\quad & \text{for all}\ c<2\sqrt{dr}.
\end{aligned}
$$

For each $u_i$ in (\ref{eq:competition_diffusion}), we denote by $c_i:=2\sqrt{d_ir_i}$ its Fisher--KPP spreading speeds. Comparison with the corresponding scalar equation shows that the $i$-th species cannot spread faster than $c_i$. Whether a species with a strictly largest Fisher--KPP speed retains this spreading speed and excludes its competitors is a central question addressed in this paper.

For the two-species Lotka--Volterra competition-diffusion system, extensive research has focused on the noncritical case, where intra- and interspecific competition differ in strength. Here, the system with normalized coefficients reads
$$
    \begin{cases}
        \partial_tu=\partial^2_x u+u(1-u-av),\\
        \partial_tv=d\partial^2_x v+rv(1-v-bu),
    \end{cases}
    \quad t>0,\ x\in \mathbb{R}.
$$
We refer to \cite{KanOn1995,Gardner1982,Girardin2019} for the strong competition case $a,b>1$, \cite{TF1980,LLL2020} for the weak competition case $0<a,b<1$, and \cite{KanOn1997,GL2019} for the strong-weak competition case $0<a<b<1$.

For the critical competition structure considered in this paper, a detailed two-species analysis was implemented by Alfaro and Xiao \cite{AlfaroXiao2023} with $q=1$. They proved the nonexistence of traveling waves connecting distinct nonzero equilibria under an \textit{ultimate monotonicity} assumption. For compactly supported initial data, they proved that the faster species(possessing a larger scalar spreading speed) excludes the slower one, and obtained refined front asymptotics involving the logarithmic Bramson correction. This fastest-species selection result contrasts with the classical phenomenon initially identified by Hastings \cite{Hastings1983} and then rigorously established by Dockery, Hutson, Mischaikow, and Pernarowski \cite{DockeryEtAl1998}: in a bounded, spatially heterogeneous environment, “slower diffusor always prevails”. This contrast illustrates how the long-time dynamics of competition-diffusion systems can differ fundamentally between bounded and unbounded domains. 

Alfaro and Xiao's analysis further revealed a \textit{bump phenomenon}: although the slower species tends uniformly to zero, its decay near the center is constrained by heat-type lower bounds. Hence, the slower species decays much slower near the center than in the outer region behind the invasion fronts. This slowly decaying bump explains the long-time coexistence of the two species observed in numerical simulations(see Fig.~\ref{figure:bump} for an illustration of the bump phenomenon). Related selections and bump results for two species in higher spatial dimensions were subsequently obtained by Cai, Ma, and Xiao \cite{CMX2025}. 

When $N\geq 3$, additional competition interactions give rise to a wider range of traveling waves and invasion patterns. Chen, Hung, Mimura, and Ueyama \cite{CHMU2012} constructed exact traveling waves and investigated the influence of an invading third species on two resident populations. Chen and Hung \cite{CH2016} used maximum-principle estimates to obtain nonexistence results of traveling profiles. Guo and Guo \cite{GG2022} studied the invasion fronts involving a stronger competitor and two weakly competing residents. Chang and Wu \cite{CW2025} investigated bistable wavefronts. Jiang and Wu \cite{JW2026} characterized spreading speeds and several propagation patterns in a three-species system with strong competition between one species and the other two. For general multispecies models, we refer to work by Lin, Li and Ma \cite{LLM2010} and work by Cantrell and Lam \cite{CL2021} for further insights. Despite these previous works, results on the long-time behavior of solutions remain limited, since the comparison principle is generally unavailable for the entire system.

\begin{figure}\centering
\begin{center}
\begin{tikzpicture}[scale = 1]
\draw[thick](-6,0) -- (6,0) node[right] {$x$};
\draw[->, thick] (0,1.5)--(0,1.9);
\draw[->, thick] (0,0.5)--(0,0.1);
\draw[dashed] (-6,2) -- (6,2);
\draw[thick, red]  (-6,0.035)--(-1,0.035)  to [out=0,in=180]  (0,0.6) to [out=0,in=180] (1,0.035)--(6,0.035);
\draw[very thick, blue] (-6,0.05)--(-5,0.05) to [out=0, in=180]  (-3,1.95)--(-0.6,1.95) to [out=0, in=180] (0,1.4) to [out=0,in=180] (0.6,1.95)--(3,1.95) to [out=0,in=180] (5,0.05)--(6,0.05);
\draw[dashed] (4.1,2.3)--(4.1,-0.3);
\draw[dashed] (-4.1,2.3)--(-4.1,-0.3);
\node[left] at (-6,2) {\footnotesize{1}};
\node[below] at (0,0) {\footnotesize{0}};
\node[above] at (2,1.5) {\Large{$v$}};
\node[above] at (2,0) {\Large{$u$}};
\node[left] at (-4.1,1) {$\mathrm{leading\ front}$\ \textbf{$\leftarrow$}};
\node[right] at (4.1,1) {\textbf{$\rightarrow$}\ $\mathrm{leading\ front}$};
\draw[dashed] (1.2,2.3)--(1.2,-0.3);
\draw[dashed] (-1.2,2.3)--(-1.2,-0.3);
\node[above] at (0,0.7) {\bf{BUMP}};
\node[left] at (-1.1,1) {\textbf{$\leftarrow$}};
\node[right] at (1.1,1) {\textbf{$\rightarrow$}};
\end{tikzpicture}
\caption{The asymptotic profile of the solution to the two-species system with $q=1$. $v$ denotes the faster species and $u$ denotes the slower one.}
\label{figure:bump}
\end{center}
\end{figure}

Significantly, the framework by Alfaro and Xiao relies on the two-species structure and does not extend directly to (\ref{eq:competition_diffusion}) with $N\geq 3$ ,even when $q=1$. 
These difficulties motivate the unified entropy approach developed in this paper.
Here we construct a family of entropy functions as weighted sums of power and logarithmic terms of the form $u_i^q/q-m_i\log u_i$. Detailed descriptions of the properties of these entropy functions are presented in Section \ref{sec:entropy method}.
By choosing the entropy weights according to the problem, we use this structure both to prove the nonexistence of traveling wave and to establish selection of the fastest species. The same approach, combined with heat-kernel estimates, gives heat-type lower bumps for every nontrivial slow species. 

\subsection{Main theorems}
We now state the main results. Our first theorem concerns bounded nonnegative traveling profile $U(x-ct)=(U_1,\ldots,U_N)(x-ct)=(u_1,\ldots,u_N)(t,x)$, with $U_i\in C^2(\mathbb{R})$, satisfying
\begin{equation}\label{eq:main_theorem_1}
    d_iU_i^{\prime\prime}+cU_i^\prime+r_i U_i\left(1-\sum_{j=1}^N U_j^q\right)=0,\quad i=1,\ldots,N,
\end{equation}
and admitting limits $U(-\infty)=A$ and $U(+\infty)=B$ in $\Sigma_q$.
\begin{theorem}[Nonexistence of traveling wave]\label{thm:main_thm1}
Let $q\geq 1$, $d_i,r_i>0$, and $c\in \mathbb{R}$. Every bounded nonnegative  profile satisfying (\ref{eq:main_theorem_1}) is constant. In particular, $A=B$.
\end{theorem}
The theorem excludes connections between distinct equilibria as well as nonconstant profiles returning to the same equilibrium. No monotonicity assumption is required, and endpoints are allowed to lie on the boundary of $\Sigma_q$. In particular, for $N=2$ and $q=1$, this removes the ultimate monotonicity condition in \cite{AlfaroXiao2023}.

We next consider the Cauchy problem for (\ref{eq:competition_diffusion}) with compactly supported initial data satisfying \begin{equation}\label{eq:sec1_theorem2_condition1}
    \quad 0\leq u_i(0,\cdot)\leq 1,\quad u_{1}(0,\cdot)\not\equiv 0.
\end{equation}
We assume that
\begin{equation}\label{eq:sec1_theorem2_condition2}
    c_1>c_s:=\max_{2\leq i\leq N}c_i,\quad c_i:=2\sqrt{d_ir_i}.
\end{equation}
No ordering among the slow speeds is required and those speeds may coincide.
\begin{theorem}[Fastest-species selection]\label{thm:main_theorem2}
    Let $q\geq 1$ and assume (\ref{eq:sec1_theorem2_condition1}) and (\ref{eq:sec1_theorem2_condition2}) hold. Then, the following holds
$$
\lim_{t\to\infty}\sum_{i=2}^N\left(\norm{u_i(t,\cdot)}_{L^p(\mathbb{R})}+\norm{u_i(t,\cdot)}_{L^\infty(\mathbb{R})}\right)=0,\quad \mathrm{for\ all}\ 1<p<\infty.
$$
$$
        \lim_{t\to\infty}\left(\sup_{\abs{x}\leq \underline{c}t}\abs{1-u_1(t,x)}+\sup_{\abs{x}\geq \overline{c}t}u_1(t,x)\right)=0,\quad \mathrm{for\ all}\ \underline{c}<c_1<\overline{c}.
$$
    Moreover, let $U_{KPP}$ be the minimal traveling wave satisfying
$$d_1U''_{KPP}+c_1U'_{KPP}+r_1U_{KPP}(1-U^q_{KPP})=0,\quad 1-U_{KPP}(-\infty)=U_{KPP}(+\infty)=0,$$
then one has
$$
\lim_{t\to\infty}\left(\sup_{x\in\mathbb{R}}\left|u_1(t,x)-U_{KPP}\left( |x|-c_1t+\frac{3d_1}{c_1}\ln t+\eta(t)\right)\right|\right)=0,
$$
where $\eta(t)$ is bounded.
\end{theorem}

By this theorem, the extinction of the slow species holds uniformly on the whole line, including the central region. However, the uniform extinction does not imply an arbitrarily rapid disappearance of the slow species. Our third result states the bump phenomenon.
\begin{theorem}[Heat-type lower bumps]\label{thm:main_theorem3}
    Under the assumptions of Theorem \ref{thm:main_theorem2}, for each $i\geq 2$ with $u_{i}(0,\cdot)\not\equiv 0$, there exist $T>0$, $R_i>0$ and $C_i>0$ such that
    \begin{equation}\label{eq:sec1_main_thm_bump_1}
        u_i(t,x)\geq \frac{C_i}{\sqrt{t-T}}\exp\left(-\frac{(|x|+R_i)^2}{4d_i(t-T)}\right),\quad \text{for all}\ t>T,\ x\in \mathbb{R}.
    \end{equation}
    In particular, for every $D>0$, there are $C_{i,D}$, $T_{i,D}>0$ such that
    \begin{equation}\label{eq:sec1_main_thm_bump_2}
        \inf_{|x|\leq D\sqrt{t}}u_i(t,x)\geq C_{i,D}t^{-\frac{1}{2}},\quad \text{for all}\ t\geq T_{i,D}.
    \end{equation}
\end{theorem}
Together with Theorem \ref{thm:main_theorem2}, this shows that every nontrivial slow species tends uniformly to zero while retaining a central bump whose amplitude cannot decay faster than the heat-type lower bound.

\subsection{The entropy method}\label{sec:entropy method} 
Our construction is inspired by the Lyapunov functions used by
Goh \cite{Goh1977} to establish global stability in generalized
Lotka--Volterra ODE systems. In \cite{Goh1977}, Goh considered
\[
\frac{\dd z_i}{\dd t}
=z_i\left(b_i+\sum_{j=1}^N a_{ij}z_j\right),
\qquad i=1,\ldots,N,
\]
which admits a strictly positive equilibrium
$z^\ast=(z_1^\ast,\ldots,z_N^\ast)$. Using the equilibrium relations
\[
b_i+\sum_{j=1}^N a_{ij}z_j^\ast=0,
\]
we can rewrite the system as
\[
\frac{\dd z_i}{\dd t}
=z_i\sum_{j=1}^N a_{ij}(z_j-z_j^\ast).
\]
To obtain a quadratic form in $z-z^\ast$ when differentiating
$V$ along solutions, we choose
\begin{equation}\label{eq:v_derivative}
\frac{\partial V}{\partial z_i}
=s_i\left(1-\frac{z_i^\ast}{z_i}\right),
\qquad s_i>0.
\end{equation}
Integrating \eqref{eq:v_derivative} and normalizing $V$ by
$V(z^\ast)=0$, we obtain
\begin{equation}\label{eq:lyapunov}
V(z)=\sum_{i=1}^N s_i
\left(
z_i-z_i^\ast-z_i^\ast\log\frac{z_i}{z_i^\ast}
\right).
\end{equation}
This function is nonnegative on $(0,\infty)^N$ and vanishes
only at $z=z^\ast$.

Set $A=(a_{ij})$ and $S=\operatorname{diag}(s_1,\ldots,s_N)$.
Along solutions,
\[
\begin{aligned}
\frac{\dd V}{\dd t}
&=\sum_{i,j=1}^N
s_i a_{ij}(z_i-z_i^\ast)(z_j-z_j^\ast)=\frac12(z-z^\ast)^\top
(SA+A^\top S)(z-z^\ast).
\end{aligned}
\]
Thus, if the positive weights $s_i$ can be chosen so that
$SA+A^\top S$ is negative definite, then
\[
\frac{\dd V}{\dd t}\le -C|z-z^\ast|^2,
\]
for some $C>0$. Moreover, $V(z)\to\infty$ whenever a component
of $z$ approaches zero or infinity. 
Since $V$ is nonincreasing along solutions, every positive
solution remains in a compact subset of $(0,\infty)^N$
and therefore exists globally. The strict negativity of
$\dd V/\dd t$ away from $z^\ast$ then implies convergence
to $z^\ast$.
This gives
Goh's sufficient condition for global asymptotic stability.

Motivated by this construction, we introduce the entropy family
\begin{equation}\label{eq:sec1_entropy_family}
\Phi_{q,m}(u)
:=\sum_{i=1}^N\frac{1}{r_i}
\left(\frac{u_i^q}{q}-m_i\log u_i\right),
\qquad
m_i\ge0,\quad \sum_{i=1}^N m_i=1,
\end{equation}
for our critical $N$-species competition--diffusion system.
When $m_i=0$, the corresponding logarithmic term is omitted.
The weights $m_i$ are chosen so that the reaction contribution
to the entropy identity becomes a negative square.

For strictly positive $m$, the connection with
\eqref{eq:lyapunov} is explicit. Taking
\[
z_i=u_i^q,\qquad z_i^\ast=m_i,
\qquad s_i=\frac{1}{qr_i},
\]
we obtain
\[
V(u_1^q,\ldots,u_N^q)
=\Phi_{q,m}(u)-\Phi_{q,m}(m^{1/q}),
\]
where $m^{1/q}=(m_1^{1/q},\ldots,m_N^{1/q})$.
Thus, up to an additive constant, $\Phi_{q,m}$ has the same
form as Goh's Lyapunov function under the change of variables
$z_i=u_i^q$. In the present critical system, however, the
entropy is used to control dissipation relative to a continuum
of equilibria, rather than to establish convergence to a
single strictly positive stable equilibrium.

For positive components, a direct calculation gives the identity 
\begin{equation}\label{eq:sec1_entropy_idenitity}
    \partial_t\Phi_{q,m}=\partial_xJ_{q,m}-\left(1-\sum_{i=1}^Nu_i^q\right)^2-\sum_{i=1}^N\frac{m_id_i}{r_i}\left(\frac{\partial_x u_i}{u_i}\right)^2-(q-1)\sum_{i=1}^N \frac{d_i}{r_i}u_i^{q-2}(\partial_x u_i)^2,
\end{equation}
where
$$
J_{q,m}:=\sum_{i=1}^N\frac{d_i}{r_i}\left(u_i^{q-1}-\frac{m_i}{u_i}\right)\partial_xu_i.
$$
Indeed, because of the constraint $\sum_{i=1}^N m_i=1$, the reaction contribution reads
$$
\sum_{i=1}^N(u_i^q-m_i)\left(1-\sum_{j=1}^Nu_j^q\right)=-\left(1-\sum_{i=1}^N u_i^q\right)^2\leq 0.
$$ 
For a traveling profile $U(x-ct)=(U_1,U_2,\ldots,U_N)(x-ct)$ with limits $U(-\infty)=A$ and $U(+\infty)=B$, (\ref{eq:sec1_entropy_idenitity}) yields that
\begin{equation}\label{eq:sec1_contradiction}
(J_{q,m}+c\Phi_{q,m})^\prime=\partial_xJ_{q,m}-\partial_t\Phi_{q,m}\geq 0.    
\end{equation}
By integrating (\ref{eq:sec1_contradiction}) over $\mathbb{R}$, we obtain an identity
\begin{equation}\label{eq:integrated identity}
    (J_{q,m}+c\Phi_{q,m})(B)-(J_{q,m}+c\Phi_{q,m})(A)=\int_{\mathbb{R}}(\partial_xJ_{q,m}-\partial_t\Phi_{q,m})\dd \xi \geq 0,
\end{equation}
whose left-hand side contains endpoint contribution and the right-hand side is the total dissipation. We choose the weight so that the entropy is minimized at $B$ when $c>0$ and at $A$ when $c<0$; when $c=0$, it vanishes. By proving that the left-hand side of (\ref{eq:integrated identity}) is always nonpositive, we derive a contradiction because the total dissipation is strictly positive for nonconstant profiles. The argument requires no monotonicity of the profiles and allows the endpoints to be boundary equilibria on $\Sigma_q$.

For the Cauchy problem, we label the species so that $c_1>c_s:=\max_{2\leq i\leq N} c_i$ and assume that the fastest species is initially nontrivial. We take $m=(1,0,\ldots,0)$ and subtract a constant, obtaining an exact nonnagative entropy
$$
    \Phi_q(u_1,u_2,\ldots,u_N):=\frac{1}{r_1}\left(\frac{u_1^q-1}{q}-\log u_1\right)+\sum_{i=2}^N\frac{u_i^q}{qr_i}.
$$
Its direct integral over the whole line is unsuitable for compactly supported initial data because $u_1$ is exponentially small far ahead of the fronts. We therefore integrate over expanding intervals $|x|\leq ct$ with intermediate speeds $c_s<c<c_1$. 

Exponential decay near $x=ct$ makes the boundary terms integrable in time (see Lemma \ref{lemma:exponential inter annuli}). The entropy then gives uniform $L^q$ bounds for the slow species and space-time control of $(1-\sum_{i=1}^Nu_i^q)^2$. Gagliardo--Nirenberg and Nash inequalities convert these bounds into $L^{2q}$ extinction, and heat-semigroup smoothing yields the uniform extinction and species selection results.   

The difference between $q>1$ and $q=1$ becomes crucial when proving the bump estimates. For $q>1$, the last term in (\ref{eq:sec1_entropy_idenitity}) supplies additional dissipation 
$$
(q-1)u_i^{q-2}\abs{\partial_xu_i}^2=\frac{4(q-1)}{q^2}|(\partial_x u_i^{q/2})^2|.
$$
It yields additional gradient controls and most importantly, makes the source terms in a Duhamel's representation integrable in suitable spatial norms(see Lemma \ref{lemma:decomposition}). Estimates for the resulting heat kernel yield
\begin{equation}\label{eq:sec1_integrability}
\int_T^\infty\|(1-\sum_{i=1}^Nu_i^q)^-(t,\cdot)\|_{L^\infty(\mathbb{R})}\dd t<\infty,
\end{equation}
with $f^-:=\max\{-f,0\}$. 

 On the other hand, when $q=1$, the additional gradient dissipation vanishes. Hence, the integrability of some terms in Duhamel's representation of $u_i$ ($i\geq 2$) remains unclear, and Lemma \ref{lemma:decomposition} does not hold in general. As will be shown in Section \ref{sec:critical case}, we fix this gap by applying the $L^{2q}$ extinction results for slow species established in Section \ref{sec:selection}. We will show that (\ref{eq:sec1_integrability}) still holds in this critical case.

Consequently, the comparison 
$$
\partial_t u_i\geq d_i\partial_x^2 u_i +r_i(1-\sum_{i=1}^Nu_i^q)^-u_i
$$
combined with (\ref{eq:sec1_integrability}) gives heat-type lower bump estimates for every slow species and every $q\geq 1$.

\noindent \textbf{Organization of the paper.} In Section \ref{sec:traveling wave}, we prove the nonexistence of traveling waves, including profiles with boundary equilibria. In Section \ref{sec:selection}, we first derive the estimates on intermediate annuli. By doing this, we establish the entropy bounds on expanding intervals and prove extinction of slow species and spreading of the fastest one. Section \ref{sec:bump1} is devoted to the heat-type lower bumps. We treat the case $q>1$ and the critical case $q=1$ separately. 
\section{Nonexistence of traveling wave}\label{sec:traveling wave}
Throughout this section, we investigate the existence of traveling wave solutions
$$
(U_1,U_2,...,U_N)(\xi)=(u_1,u_2,...,u_N)(t,x),\qquad \xi=x-ct,
$$
whose profiles $U_i\in C^2(\mathbb{R})$ are bounded, nonnegative and solves
\begin{equation}\label{eq:traveling wave}
\begin{cases}
    d_iU_i^{\prime\prime}+cU_i^\prime+r_iU_iW_q=0,\\
    (U_1,U_2,...,U_N)(-\infty)=A,\\
     (U_1,U_2,...,U_N)(+\infty)=B,
\end{cases}
\qquad W_q:=1-\sum_{i=1}^NU_i^q,
\end{equation}
where $A=(a_1,a_2,\dots,a_N)$ and $B=(b_1,b_2,\dots,b_N)$ lie on $\Sigma_q$ which is defined by (\ref{eq:sec1_sigma_q}). 

Firstly, it can be seen that for each $i$, the following properties holds.

\noindent (P1) Either $U_i\equiv 0$ or $U_i(\xi)>0$ for every $\xi\in \mathbb{R}$.

\noindent (P2) Both $U^\prime_i$ and $U_i^{\prime\prime}$ are bounded on $\mathbb{R}$. Moreover, $U_i^\prime\to 0$ as $\xi\to\pm\infty$.

We prove that every profile satisfying (\ref{eq:traveling wave}) is constant. The proof uses the integrated dissipation identity (\ref{eq:integrated identity}) associated with the entropy functions (\ref{eq:sec1_entropy_family}). To apply this identity, we show that the entropy flux $J_{q,m}$ vanishes at both ends. Because the flux contains $U_i^\prime/U_i$, the case when a component $U_i$ tends to zero requires careful treatment. We address this issue by choosing a suitable entropy weight. Finally, comparing the signs of both sides of (\ref{eq:integrated identity}) then forces the profile to be constant.
\subsection{Preliminaries}
By (P1), without loss of generality, we assume $U_i>0$ for every $i=1,2,\ldots,N$. For $q\geq 1$ and a weight vector $m=(m_1,m_2,...,m_N)$ satisfying
\begin{equation}\label{eq:m}
m_i\geq 0,\qquad \sum_{i=1}^Nm_i=1,
\end{equation}
the entropy (\ref{eq:sec1_entropy_family}) evaluated along the traveling profile is rewritten as
$$
\Phi_{q,m}(U_1,U_2,\ldots,U_N)=\sum_{i=1}^N\frac{1}{r_i}\left(\frac{1}{q}U_i^q-m_i\log U_i\right).
$$
The corresponding flux is
$$
J_{q,m}=\sum_{i=1}^N\frac{d_i}{r_i}U_i^{q-1}U_i^\prime-\sum_{i=1}^N\frac{m_id_i}{r_i}\frac{U_i^\prime}{U_i}.
$$
For a traveling wave profile, $\partial_t\Phi_{q,m}=-c\Phi_{q,m}^\prime$ and $\partial_x J_{q,m}=J_{q,m}^\prime$. Consequently, a direct calculation yields that
\begin{equation}\label{eq:entropy identity}
    \frac{\dd}{\dd \xi}(J_{q,m}+c\Phi_{q,m})=W_q^2+\sum_{i=1}^N\frac{m_id_i}{r_i}(\frac{U_i^\prime}{U_i})^2+(q-1)\sum_{i=1}^N \frac{d_i}{r_i}U_i^{q-2}(U_i^\prime)^2.
\end{equation}
Since $q\geq 1$, every term on the right-hand side is non-negative. Thus, $J_{q,m}+c\Phi_{q,m}$ is nondecreasing in $\xi$. 

To use this monotonicity, we need to control the flux at both endpoints of the profile. The next lemma provides a uniform bound on $|{U_i^\prime(\xi)}/{U_i(\xi)}|$
and shows that they vanish at both ends when $c=0$, even when
the corresponding component tends to zero.
\begin{lemma}\label{lemma:bounded log}
    For every $i$ with $U_i\not \equiv 0$,
    \begin{equation}\label{eq:bounded_V}
        \sup_{\xi\in \mathbb{R}}\abs{\frac{U_i^\prime(\xi)}{U_i(\xi)}}<\infty.
    \end{equation}
    If $c=0$, then, regardless of whether $U_i$ vanishes at endpoints,
    \begin{equation}\label{eq:bounded_V_c0}
        \frac{U_i^\prime(\xi)}{U_i(\xi)}\to 0,\quad as\ \xi\to \pm\infty.
    \end{equation}
\end{lemma}
\begin{proof}
Set $V_i=U_i^\prime/U_i$. We can obtain that
\begin{equation}\label{eq:v_i}
V^\prime_i=-V_i^2-\frac{c}{d_i}V_i-\frac{r_i}{d_i}W_q.
\end{equation}
It follows from the boundedness of $W_q$ that there exists some $M>0$ such that 
\begin{equation}\label{eq:logrithmic_d}
    V^\prime_i\leq  -\frac{1}{2}V_i^2,\quad \text{whenever}\ \abs{V_i}\geq M.
    \end{equation}
Suppose first that $V_i(\xi_0)>M$ for some $\xi_0$. (\ref{eq:logrithmic_d}) implies that $V_i(\xi)>M$ for all $\xi\leq\xi_0$. 
On this interval, $(1/V_i)^\prime=-(V_i^\prime/V_i^2)\geq 1/2$. Hence,
    $$
    \frac{1}{V_i(\xi)}\leq \frac{1}{2}(\xi-\xi_0)+\frac{1}{V_i(\xi_0)},\quad \xi\leq \xi_0
    $$
and $V_i$ becomes negative at a finite distance to the left of $\xi_0$, contradicting $V_i(\xi)>M$. A similar contradiction can be obtained when $V_i(\xi_0)<-M$ for some $\xi_0$. Consequently, $\abs{V_i(\xi)}\leq M$ for every $\xi\in \mathbb{R}$ and (\ref{eq:bounded_V}) holds.

    Assume now that $c=0$. Then, (\ref{eq:v_i}) becomes 
    $$
    V^\prime_i=-V_i^2-\frac{r_i}{d_i}W_q.
    $$
    Since $W_q(\xi)\to 0$ as $\xi\to+\infty$, for every $\varepsilon>0$ there is a large enough $\xi^\ast$ such that
    $$
   \left |\frac{r_i}{d_i}W_q\right|\leq \frac{\varepsilon^2}{4},\quad \text{for all}\ \xi\geq \xi^\ast.
    $$
    If $V_i(\xi_0)\leq-\varepsilon$ at some $\xi_0\geq \xi^\ast$, then $V_i$ remains in this region for $\xi\geq \xi_0$ since  
    $$
    V_i^\prime(\xi)\leq -V_i^2+\frac{\varepsilon^2}{4}\leq -\frac{3}{4}V_i^2.
    $$
    It follows that $V_i$ blows up in finite $\xi$, which derives a contradiction. Based on the above inequality, we can also prove $V_i(\xi)<\varepsilon$ for all sufficiently large $\xi$. Since $\varepsilon$ is arbitrary, $V_i(+\infty)=0$ holds. The proof for $\xi\to-\infty$ is analogous.
\end{proof}
\begin{remark}
    For $q\geq 1$, property (P2) implies $U_i^{q-1}U_i^{\prime}$ and $U_i^\prime$ tend to zero at both ends. Hence, if an endpoint value is positive, $J_{q,m}\to 0$ there. However, at a zero endpoint, boundedness of $U_i^\prime/U_i$ alone does not give this limit. (\ref{eq:bounded_V_c0}) solves this issue for $c=0$. For $c\neq 0$, we will solve the issue by choosing weights $m_i$ whose positive entries correspond to components that are positive at both endpoints (see the next subsection for details).  
\end{remark}
\subsection{Proof of Theorem \ref{thm:main_thm1}}
\begin{proposition}\label{prop:standing wave}
If $c=0$, every bounded nonnegative  profile satisfying (\ref{eq:traveling wave}) is constant. In particular, $A=B$.
\end{proposition}
\begin{proof}
    Choose any $m$ satisfying (\ref{eq:m}). With $c=0$, (\ref{eq:entropy identity}) becomes
    \begin{equation}\label{eq:proof c0}
    \frac{\dd}{\dd \xi}J_{q,m}=W_q^2+\sum_{i=1}^N\frac{m_id_i}{r_i}\left(\frac{U_i^\prime}{U_i}\right)^2+(q-1)\sum_{i=1}^N \frac{d_i}{r_i}U_i^{q-2}(U_i^\prime)^2\geq 0.
    \end{equation}
    Property (P2) and Lemma \ref{lemma:bounded log} yield $J_{q,m}(\pm\infty)=0$. Hence, by (\ref{eq:proof c0}), $J_{q,m}=J^\prime_{q,m}\equiv 0$. In particular, $W_q\equiv 0$. The equation of $U_i$ reduces to $d_iU_i^{\prime\prime}=0$, and boundedness of $U_i$ on $\mathbb{R}$ implies every component to be constant. 
\end{proof}
Next we consider $c>0$ and choose the entropy weights from the right endpoint:
\begin{equation}\label{eq:m_c_neq0}
    m_i=b_i^q,\quad i=1,2,\dots,N.
\end{equation}
These weights satisfy (\ref{eq:m}). The next lemma shows that every component with positive weight also has a positive limit at the left endpoint. This ensures that the entropy function has finite limits and the flux vanishes at both ends.
\begin{lemma}\label{lemma:finite entropy}
    If a profile satisfying (\ref{eq:traveling wave}) exists with $c>0$, then for any $i$, $b_i>0$ implies $a_i>0$.  
\end{lemma}
\begin{proof}
     Use the weights (\ref{eq:m_c_neq0}). It can be seen that
     $$
     J_{q,m}+c\Phi_{q,m}=\sum_{i=1}^N\frac{d_i}{r_i}U_i^{q-1}U_i^\prime-\sum_{i=1}^N\frac{b_i^qd_i}{r_i}\frac{U_i^\prime}{U_i}+\sum_{i=1}^N\frac{c}{r_i}(\frac{1}{q}U_i^q-b_i^q \log U_i)
     $$
     is non-decreasing in $\xi$ by \eqref{eq:entropy identity}. Note that, at $+\infty$, Property (P2) gives 
     $$
     J_{q,m}(+\infty)=0,\quad(J_{q,m}+c\Phi_{q,m})(+\infty)=c\Phi_{q,m}(B)<\infty.
     $$
     Suppose that the right endpoint $b_k>0$ but the left endpoint $a_k=0$ for some $k$. Then $-b_k^q\log U_k(\xi)\to +\infty$ as $\xi\to-\infty$. It follows from Property (P2) and Lemma \ref{lemma:bounded log} that 
     $$
     J_{q,m}+c\Phi_{q,m}\to \infty,\quad as\ \xi\to-\infty,
     $$
     which contradicts the monotonicity of $J_{q,m}+c\Phi_{q,m}$ along the traveling profile. Thus $a_k>0$.
\end{proof}
\begin{remark}
    For $c<0$, reflection $\tilde{U}(\xi)=U(-\xi)$ changes the speed to $-c>0$ and connects $B$ to $A$. Applying Lemma \ref{lemma:finite entropy} to $\tilde{U}$ gives a symmetric result showing that $a_i>0$ implies $b_i>0$. 
\end{remark}
We can now compare the entropy at the two endpoints. For the weights $m_i=b_i^q$, the entropy is minimized at the right endpoint $B$, whereas the entropy identity requires it to increase along a profile with $c>0$. These two properties force the traveling wave to be constant.
\begin{proposition}\label{prop:general tw}
    If $c\neq 0$, every bounded nonnegative profile satisfying (\ref{eq:traveling wave}) is constant. In particular, $A=B$.
\end{proposition}
\begin{proof}
It suffices to treat $c>0$. Choose $m_i=b_i^q$. It can be seen from Lemma \ref{lemma:finite entropy} that
$$
(J_{q,m}+c\Phi_{q,m})(\pm\infty)<\infty
$$
since $m_i=0$ if $a_i=0$. Moreover, we can obtain that
$$
J_{q,m}(\pm\infty)=0.
$$
   Therefore, integrating (\ref{eq:entropy identity}) over $\mathbb{R}$ yields 
    \begin{equation}\label{eq: contradiction}
        c[\Phi_{q,m}(B)-\Phi_{q,m}(A)]=\int_\mathbb{R}W_q^2\dd \xi +\sum_{i=1}^N\int_\mathbb{R}\frac{d_i}{r_i}\left[\frac{b_i^q}{U_i^2}+(q-1)U_i^{q-2}\right](U_i^\prime)^2\dd \xi.
    \end{equation}
    It is clear that the right-hand side of \eqref{eq: contradiction} is strictly positive. Otherwise, it gives $W_q\equiv 0$ and the profile is a constant function. 

    We claim that the left-hand side of \eqref{eq: contradiction} is nonpositive. Indeed, it follows that 
    $$
    \begin{aligned}
        \Phi_{q,m}(B)-\Phi_{q,m}(A)&=\sum_{\substack{1\leq i\leq N\\ b_i>0}}\frac{1}{r_i}\left[\left(\frac{b_i^q}{q}-b_i^q\log b_i\right)-\left(\frac{a_i^q}{q}-b_i^q\log a_i\right)\right]-\sum_{\substack{1\leq i\leq N\\ b_i=0}}\frac{a_i^q}{qr_i}\\
        &=\sum_{\substack{1\leq i\leq N\\ b_i>0}}\frac{b_i^q}{qr_i}\left(1-\frac{a_i^q}{b_i^q}+\log \frac{a_i^q}{b_i^q}\right)-\sum_{\substack{1\leq i\leq N\\ b_i=0}}\frac{a_i^q}{qr_i}\leq 0,
    \end{aligned}
    $$
    since the inequality $1-x+\log x\leq 0$ holds for $x>0$, with equality if and only if $x=1$. 
    This derives a contradiction. Hence, there is no nonconstant traveling wave profile with $c>0$. The argument for $c<0$ is symmetric. 
\end{proof}

\noindent\textbf{\textit{Proof of Theorem \ref{thm:main_thm1}}.}
The first main theorem follows from Proposition \ref{prop:standing wave} and Proposition \ref{prop:general tw}.

\begin{remark}
    This conclusion holds for any finite number of species and require no monotonicity assumption on the profiles. 
\end{remark}

\section{The fastest-species selection}\label{sec:selection}
Throughout the rest of this paper, we consider the Cauchy problem with compactly supported initial data. Moreover, we assume (\ref{eq:sec1_theorem2_condition1}) and (\ref{eq:sec1_theorem2_condition2}) hold.
Henceforth, we denote 
$$
W_q(t,x):=1-\sum_{i=1}^N u_i^q(t,x).
$$
We first establish extinction of the slow species in $L^{2q}(\mathbb{R})$ and $L^\infty(\mathbb{R})$, and then determine the spreading behavior of $u_1$. The proof of Theorem \ref{thm:main_theorem2} is given at the end of the section, while the proof of extinction in $L^p(\mathbb{R})$ for every $1<p<\infty$ requires an additional estimate from Section \ref{sec:bump1}. This conclusion is stated at the end of this paper.  
\subsection{Preliminaries}
We introduce some elementary bounds and an exact entropy that will play crucial roles in the sequel. Firstly, the standard parabolic comparison principle gives for every $t>0$,
$$
\begin{aligned}
    0<u_1(t,x)\le1,\quad 0\le u_i(t,x)\le1\quad (i\geq 2).
\end{aligned}
$$
Every component with nonzero initial data is strictly positive for $t>0$. Since $W_q\leq 1$, each component satisfies 
\begin{equation}
\label{eq:supress competition}
 \partial_t u_i\le d_i \partial^2_xu_i+r_iu_i. 
\end{equation}
Comparison with the linear equation gives the following exterior estimates.
\begin{lemma}\label{lemma:gauss tail}
    Suppose that $\supp u_i(0,\cdot)\subset[-R,R]$ for some $R>0$. There is a constant $C$ such that
    \begin{equation}\label{eq:gauss tail}
        u_i(t,x)\leq Ct^{-\frac{1}{2}}\exp\left(r_i t-\frac{(|x|-R)^2}{4d_i t}\right),\quad \mathrm{for\ all}\ t>0,\ |x|\geq R. 
    \end{equation}
    Hence, for every $c>c_i$, there are $C,\kappa>0$ such that 
    \begin{equation}\label{eq:exponentially small}
        \sup_{|x|>ct}u_i(t,x)\leq Ce^{-\kappa t},
    \end{equation}
    and,
    \begin{equation}\label{eq:exterior tail}
        \int_{|x|>ct}(u_i+u_i^q+u_i^{2q})(t,x)\dd x\leq Ce^{-\kappa t},
    \end{equation}
    for all sufficiently large $t$.
\end{lemma}
\begin{proof}
    (\ref{eq:gauss tail}) and (\ref{eq:exponentially small}) directly follow from (\ref{eq:supress competition}) and standard parabolic comparison arguments. Since $0\leq u_i\leq 1$ and $q\geq 1$, $u_i^{2q}\leq u_i^q\leq u_i$. Integration of the Gaussian tail estimate (\ref{eq:gauss tail}) over $\abs{x}>ct$ gives (\ref{eq:exterior tail}).
\end{proof}

We choose an exact entropy from the family (\ref{eq:sec1_entropy_family}) with weight vector $m=(1,0,\ldots,0)$ and normalize it to vanish at $(1,0,\ldots,0)$:
$$
    \Phi_q(u_1,u_2,\ldots,u_N):=\frac{1}{r_1}\left(\frac{u_1^q-1}{q}-\log u_1\right)+\sum_{i=2}^N\frac{u_i^q}{qr_i}.
$$
Since $x-1-\log x\geq 0$ for $x>0$, we can obtain that
\begin{equation}\label{eq:positive entropy}
    \Phi_q\geq \sum_{i=2}^N\frac{u_i^q}{qr_i}\geq 0.
\end{equation}
\begin{lemma}
    For $t>0$ and $q\geq 1$, 
    \begin{equation}\label{eq:entropy_time derivative} 
        \partial_t\Phi_q=\partial_x J_q-W_q^2-\frac{d_1}{r_1}(\partial_x\log u_1)^2-\sum_{i=1}^N\frac{d_i}{r_i}(q-1)u_i^{q-2}(\partial_x u_i)^2,
    \end{equation}
    where 
    \begin{equation}\label{eq:j_q}
        J_q:=-\frac{d_1}{r_1}\partial_x\log u_1+\sum_{i=1}^N\frac{d_i}{r_i}u_i^{q-1}\partial_x u_i.
    \end{equation}
    For $q=1$ the last term in (\ref{eq:entropy_time derivative}) is absent, and
    $$
        \partial_t\Phi_1=-\partial_x^2\left(\frac{d_1}{r_1}\log u_1-\sum_{i=1}^N\frac{d_i}{r_i}u_i\right)-\left(1-\sum_{i=1}^Nu_i\right)^2-\frac{d_1}{r_1}(\partial_x\log u_1)^2.
    $$
\end{lemma}
\begin{proof}
    Differentiating the entropy function gives
    $$
    \partial_{u_1}\Phi_q=\frac{1}{r_1}\left(u_1^{q-1}-\frac{1}{u_1}\right),\quad \partial_{u_i}\Phi_q=\frac{1}{r_i}u_i^{q-1}\quad (i\geq 2).
    $$
    Hence, the reaction terms contribute exactly
    $$
    \sum_{i=1}^N(\partial_{u_i}\Phi_q)(r_iu_iW_q)=(\sum_{i=1}^Nu_i^q-1)W_q=-W_q^2.
    $$
    The product rule gives 
$$
u_i^{q-1}\partial_x^2u_i
=\partial_x\left(u_i^{q-1}\partial_xu_i\right)
-(q-1)u_i^{q-2}(\partial_xu_i)^2\quad \text{and}\quad \frac{\partial_x^2u_1}{u_1}=\partial_x \left(\frac{\partial_x u_1}{u_1}\right)+\left(\frac{\partial_x u_1}{u_1}\right)^2.
$$
Consequently, for diffusion terms of $\partial_t \Phi_{q,m}$,
$$
\begin{aligned}
\sum_{i=1}^N d_i(\partial_{u_i}\Phi_q)\partial_x^2u_i
&=\partial_x\left(
-\frac{d_1}{r_1}\partial_x\log u_1
+\sum_{i=1}^N\frac{d_i}{r_i}u_i^{q-1}\partial_xu_i
\right)\\
&\quad-\frac{d_1}{r_1}(\partial_x\log u_1)^2
-\sum_{i=1}^N\frac{d_i}{r_i}(q-1)
u_i^{q-2}(\partial_xu_i)^2.
\end{aligned}
$$
The expression inside the divergence is precisely $J_q$. Combining this identity with the reaction contribution proves (\ref{eq:entropy_time derivative}).
\end{proof}
\begin{remark}
For $q>1$, the entropy dissipation contains the additional term 
$$
(q-1)u_i^{q-2}(\partial_x u_i)^2=\frac{4(q-1)}{q^2}|\partial_x(u_i^{\frac{q}{2}})|^2.
$$
Although this term does not affect the subsequent extinction arguments, the extra control 
plays a crucial role in establishing the lower bounds of the bump in Section \ref{sec:bump1}.
\end{remark}
\subsection{Exponential control on intermediate annuli}
Fix an intermediate speed 
$$
    c_s<c<c_1.
$$
Lemma \ref{lemma:gauss tail} implies that the slow species are exponentially small near $x=\pm ct$. We next show that $u_1$ converges exponentially to $1$ there, and its spatial derivatives converges exponentially to $0$. These estimates will control the boundary terms in the entropy identity on the expanding interval $[-ct,ct]$

\begin{lemma}\label{lemma:u upper bound}
    For every $c_s<a<b<c_1$, there are $M,T>0$ such that
    $$
        u_1(t,x)\geq M,\quad at\leq \abs{x}\leq bt,\quad t\geq T.
    $$
\end{lemma}
\begin{proof}
We prove only the positive side ($at<x<bt$), while the proof of the negative side is almost identical. For a larger annulus $\underline{a}\leq|x|\leq \overline{b}$ with $c_s<\underline{a}<a<b<\overline{b}<c_1$, Lemma \ref{lemma:gauss tail} gives 
$$
        \sum_{i=2}^Nu_i^q(t,x)\leq Ce^{-\kappa t},\quad \text{for all}\ \underline{a}t\leq |x|\leq \overline{b}t,
$$
for some $C,\kappa>0$ and large $t$. We choose some $\varepsilon$ sufficiently small such that $b<2\sqrt{(1-\varepsilon)d_1r_1}<c_1$. The most crucial step is to construct a family of sub-solutions of the following equation
$$
    \partial_t v=d_1\partial^2_{x} v+r_1(1-\varepsilon)v
$$
on proper expanding intervals. With this step completed, the parabolic comparison principle on the moving interval gives proper lower bounds of $u_1$.

Indeed, for any $\sigma\in[a,b]$, we define
$$
\phi_\sigma(y):=e^{-\frac{\sigma y}{2d_1}}\cos\left(\frac{\pi y}{2L}\right),\quad -L\leq y\leq L.
$$
It is clear that $\phi_\sigma>0$ in $(-L,L)$ and $\phi_\sigma(\pm L)=0$.
We choose $L>0$ large enough so that
$$
\left(4(1-\varepsilon)r_1-\frac{\sigma^2}{d_1}\right)L^2>d_1\pi^2.
$$
Then, a direct calculation yields 
$$
d_1\phi_\sigma^{\prime\prime}+\sigma\phi_\sigma'+(1-\varepsilon)r_1\phi_\sigma=\left((1-\varepsilon)r_1-\frac{\sigma^2}{4d_1}-\frac{d_1\pi^2}{4L^2}\right)\phi_\sigma>0.
$$
We set 
$$
v_\sigma(t,x):=M\phi_\sigma(x-\sigma t),\qquad x\in I_\sigma(t):=[\sigma t-L,\sigma t+L],
$$
for some $M>0$ that will be determined later. It can be seen that
$$\begin{aligned}
    \partial_t v_\sigma-d_1\partial^2_x v_\sigma&=-M(d_1\phi_\sigma^{\prime\prime}+\sigma\phi_\sigma')\leq(1-\varepsilon)r_1v_\sigma.
\end{aligned}$$
Choose $\eta>0$ such that $\eta+\eta^q<\varepsilon$. By increasing $T$ if necessary, we obtain that for every $t\geq T$ and $\sigma\in[a,b]$, $I_\sigma(t)\subset[\underline{a}t,\overline{b}t]$ and
$$
\sum_{i=2}^Nu_i^q(t,x)\leq\eta,\quad \text{for all}\ x\in I_\sigma(t).
$$
We let $M=e^{-\frac{bL}{2d_1}}\min\{\eta,N\}$ with $N:=\min_{x\in[aT-L,bT+L]}u_1(T,x)$. It follows that
$$
v_\sigma(T,x)\leq u_1(T,x),\quad \text{for all}\ aT-L\leq x\leq bT-L
$$
and $v_\sigma\leq \eta$. It implies that 
$$
\partial_t v_\sigma-d_1\partial^2_x v_\sigma\leq (1-\varepsilon)r_1v_\sigma\leq r_1v_\sigma\left(1-v_\sigma^q-\sum_{i=2}^N u_i^q(t,x)\right),\ \text{for all}\ t>T.
$$
It is clear that $v_\sigma(t,\sigma t\pm L)=0<u_1(t,x)$ for every $t>0$. By the parabolic comparison principle in $I(t)$, 
$$
u_1(t,\sigma t)\geq M \phi_\sigma(0)=M.
$$
Since $\sigma$ is arbitrary in $[a,b]$, the conclusion holds.
\end{proof}
\begin{lemma}\label{lemma:exponential inter annuli}
    For every $c_s<c<c_1$, there are $\delta,C,\kappa,T>0$ such that
    \begin{equation}\label{eq:exponential inter annuli}
        \sup_{\abs{|x|-ct}\leq \delta t}\left(|1-u_1|+\sum_{i=2}^Nu_i+\sum_{i=1}^N|\partial_x u_i|\right)(t,x)\leq C e^{-\kappa t},\quad \mathrm{for\ all}\ t\geq T.
    \end{equation}
\end{lemma}
\begin{proof}
Choose $\delta>0$ so that $c_s<c-2\delta<c+2\delta<c_1$. Let $w:=1-u_1$. It follows that
$$
    \partial_t w=d_1\partial_x^2 w-r_1u_1(1-u_1^q)+r_1u_1\sum_{i=2}^Nu_i^q.
$$
By the preceding Lemma \ref{lemma:u upper bound}, there exists $M>0$ such that $u_1\geq M$ on $(c-2\delta)t\leq x\leq (c+2\delta)t$ for large $t$. Since $1-u_1^q\geq 1-u_1=w$, Lemma \ref{lemma:gauss tail} implies there are $\kappa^\ast,C_2,T>0$ such that
$$
    \partial_t w- d_1\partial_x^2 w+C_1w\leq C_2e^{-\kappa^\ast t},\quad \text{for all}\ (c-2\delta)t\leq x\leq (c+2\delta)t,\ t\geq T,
$$
where $C_1=r_1M$. We define 
$$
\overline{w}(t,x)=M_1e^{-\lambda(t-T)}+M_2 e^{-\eta[x-(c-2\delta)t]}+M_3e^{-\eta[(c+2\delta)t-x]}
$$
for some positive constants $\eta,M_i$ ($i=1,2,3$) and some small enough $\lambda$ such that
$$
0<\lambda<\min\{C_1,\kappa^\ast,\delta\eta\}.
$$ 
We can obtain that
$$
\begin{aligned}
(\partial_t- d_1\partial_x^2 +C_1)\overline{w}(t,x)&=(C_1-\lambda)M_1e^{-\lambda(t-T)}+[C_1+(c-2\delta)\eta-d_1\eta^2]M_2e^{-\eta[x-(c-2\delta)t]}\\
&\qquad+[C_1-(c+2\delta)\eta-d_1\eta^2]M_3e^{-\eta[(c+2\delta)t-x]}.
\end{aligned}
$$
Let $\eta$ be sufficiently small so that 
$$
C_1+(c-2\delta)\eta-d_1\eta^2>0,\qquad C_1-(c+2\delta)\eta-d_1\eta^2>0.
$$ 
It follows that
$$
(\partial_t- d_1\partial_x^2 +C_1)\overline{w}(t,x)\geq C_2e^{-\kappa^\ast t},\qquad (c-2\delta)t\leq x\leq (c+2\delta)t,
$$
for large $t$. Choose $M_1$ large enough that $\overline{w}(T,x)\geq \omega(T,x)$ for every $x\in[(c-2\delta)T,(c+2\delta)T]$ and choose $M_2,M_3\geq 1$ so that $\overline{w}(t,(c\pm2\delta)t)\geq 1\geq w(t,(c\pm2\delta)t)$ for $t>0$. The comparison principle gives the exponential convergence of $w$ to $0$ on intermediate annuli around $x=ct$ with $c_s<c<c_1$. By Lemma \ref{lemma:gauss tail}, exponential decay of slow species $u_i$ ($i\geq 2$) holds.

Standard interior parabolic estimates, followed by Sobolev embedding, give the same exponential rate for $\partial_x u_i$ ($i\geq1$). This completes the proof of (\ref{eq:exponential inter annuli}).
\end{proof}
\subsection{$L^{2q}$ and $L^{\infty}$ extinction of slow species}
Fix $c_s<c<c_1$ and define the entropy on the expanding interval by
$$
    E_c(t):=\int_{-ct}^{ct}\Phi_q(t,x)\dd x. 
$$
Differentiating with respect to $t$ and using (\ref{eq:entropy_time derivative}), we obtain 
\begin{equation}\label{eq:d_e}
\begin{aligned}
    \frac{\dd}{\dd t}E_c(t)&+\int_{-ct}^{ct}W_q^2\dd x+\frac{d_1}{r_1}\int_{-ct}^{ct}(\partial_x\log u_1)^2\dd x+D_q(t)\\
    &=J_q(t,ct)-J_q(t,-ct)+c\Phi_q(t,ct)+c\Phi_q(t,-ct),
\end{aligned}
\end{equation}
where $D_1(t):=0$ and for $q>1$,
$$
    D_q(t):=(q-1)\sum_{i=1}^N \frac{d_i}{r_i}\int_{-ct}^{ct}u_i^{q-2}(\partial_xu_i)^2\dd x\geq 0
$$
and $J_q$ is defined by (\ref{eq:j_q}). It is easy to verify that as $u_1$ tends to $1$,
$$
\frac{u_1^q-1}{q}-\log u_1=\frac{q}{2}(1-u_1)^2+\mathcal{O}(|1-u_1|^3),\quad u_1^{q-1}-\frac{1}{u_1}=\mathcal{O}(|1-u_1|).
$$
Using Lemma \ref{lemma:exponential inter annuli} and the expression for $J_q$, it can be seen that the absolute value of the right-hand side of (\ref{eq:d_e}) is exponentially small for large $t$. More precisely, there exist $\kappa,C,T>0$ such that 
\begin{equation}\label{eq:small j}
    |J_q(t,ct)-J_q(t,-ct)+c\Phi_q(t,ct)+c\Phi_q(t,-ct)|\leq Ce^{-\kappa t},\quad \mathrm{for\ all}\ t>T.
\end{equation}
Then, for any intermediate speed $c$, we can obtain the following bounds inside the expanding cone $|x|\leq ct$.
\begin{proposition}\label{prop:Linfty_L1}
    For every $c_s<c<c_1$, there exists $T>0$ such that
    \begin{equation}\label{eq:bound_1}
    \sup_{t\geq T}E_c(t)<\infty,
    \end{equation}
    \begin{equation}\label{eq:bound_2}
    \int_T^\infty\int_{-ct}^{ct}W_q^2\dd x\dd t<\infty,
    \end{equation}
    and
    \begin{equation}\label{eq:bound_3}
        \sup_{t\geq T}\int_{\mathbb{R}}u_i^q(t,x)\dd x<\infty,\quad i=2,3,...,N.
    \end{equation}
    If $q>1$, then also
    \begin{equation}\label{eq:bound_4}
    \sum_{i=1}^N\int_T^\infty\int_{-ct}^{ct} u_i^{q-2}(\partial_x u_i)^2\dd x \dd t<\infty.
    \end{equation}
\end{proposition}
\begin{proof}
    Integrating (\ref{eq:d_e}) from large enough $T$ to $t$ and using (\ref{eq:small j}), we find 
    $$
    E_c(t)-E_c(T)+\int_T^t\int_{-cs}^{cs}\left[W_q^2+\frac{d_1}{r_1}\left(\frac{\partial_x u_1}{u_1}\right)^2(s,x)\right]\dd x\dd s+\int_T^tD_q(s)\dd s<C,
    $$
    for some $C>0$. Then (\ref{eq:bound_1}), (\ref{eq:bound_2}) and (\ref{eq:bound_4}) follow from positivity of $\Phi_q$ and $D_q$. By (\ref{eq:positive entropy}), we obtain that
    $$
    \sup_{t\geq T}\sum_{i=2}^N\frac{1}{qr_i}\int_{-ct}^{ct}u_i^q(t,x)\dd x\leq \sup_{t\geq T} E_c(t)<\infty.
    $$
    Outside the expanding interval, Lemma \ref{lemma:gauss tail} applies since $c>c_s$. This completes the proof of (\ref{eq:bound_3}).
\end{proof}

By combining the bounds obtained from Proposition \ref{prop:Linfty_L1} with the one-dimensional Gagliardo--Nirenberg and Nash inequalities, we can establish the $L^{2q}$ extinction of $u_i$ for $i\geq 2$. 
\begin{proposition}[$L^{2q}$ extinction of slow species]\label{prop:L2q extinction}
Under assumptions (\ref{eq:sec1_theorem2_condition1}) and (\ref{eq:sec1_theorem2_condition2}),  
    \begin{equation}\label{eq:L2q extinction}
    \sum_{i=2}^N\norm{u_i(t)}_{L^{2q}(\mathbb{R})}\to 0,\quad \mathrm{as}\ t\to\infty.\end{equation}
\end{proposition}
\begin{proof}
    For given $i\geq 2$, the equation of $u_i$ reads
    $$
    \partial_t u_i=d_i \partial^2_x u_i+r_iu_i W_q,\quad t>0,\ x\in \mathbb{R}.
    $$
    Multiplying the above equation by $u_i^{2q-1}$ and integrating over $\mathbb{R}$ yields
    $$
    \begin{aligned}
        \frac{1}{2q}\frac{\dd}{\dd t}\int_\mathbb{R}u_i^{2q}\dd x&=-d_i (2q-1)\int_\mathbb{R}u_i^{2q-2}(\partial_x u)^2\dd x+r_i\int_\mathbb{R}u_i^{2q}W_q\dd x\\
        &=-\frac{d_i(2q-1)}{q^2}\norm{\partial_x(u_i^q)}_{L^2(\mathbb{R})}^2+r_i\int_\mathbb{R}u_i^{2q}W_q\dd x.
    \end{aligned}
    $$
    Equivalently,
    \begin{equation}\label{eq:L2q_eq1}
        \frac{\dd}{\dd t}\norm{u_i^{q}(t)}_{L^2(\mathbb{R})}^2+\frac{2d_i(2q-1)}{q}\norm{\partial_x(u_i^q)}_{L^2(\mathbb{R})}^2=2qr_i\int_\mathbb{R}u_i^{2q}W_q\dd x.
    \end{equation}
    We fix $c_s<c<c_1$. Inside the cone $|x|\leq ct$, it follows from Hölder's inequality and the one-dimensional Gagliardo--Nirenberg inequality that
    $$
    \begin{aligned}
        \abs{\int_{-ct}^{ct}u_i^{2q}W_q\dd x}&\leq \norm{W_q}_{L^2(|x|\leq ct)}\norm{u_i^q}_{L^4(\mathbb{R})}^2\\
        &\leq C_1\norm{W_q}_{L^2(|x|\leq ct)}\norm{\partial_x(u_i^q)}_{L^2(\mathbb{R})} \norm{u_i^q}_{L^1(\mathbb{R})}\\
        &\leq C_2\norm{W_q}_{L^2(|x|\leq ct)}\norm{\partial_x(u_i^q)}_{L^2(\mathbb{R})}.
    \end{aligned}
    $$
    Here, the Gagliardo--Nirenberg constant $C_1$ is independent of $u_i^q$, and thus independent of $t$. The last line follows from (\ref{eq:bound_3}), where $C_2$ remains time-independent. Young's inequality therefore yields 
    $$
    2qr_i\abs{\int_{-ct}^{ct}u_i^{2q}W_q\dd x}\leq \frac{d_i (2q-1)}{q}\norm{\partial_x(u_i^q)}_{L^2(\mathbb{R})}^2+C_3\norm{W_q}_{L^2(|x|\leq ct)}^2,
    $$
    for some positive constant $C_3$. 

    Outside the cone, boundedness of $\abs{W_q}$ and Lemma \ref{lemma:gauss tail} imply 
    \begin{equation}\label{eq:L2q_eq3}
        \int_{|x|>ct}u_i^{2q}|W_q|\dd x\leq C_3 e^{-\kappa t},
    \end{equation}
    for some $\kappa>0$ and large $t$.
    
    Set $Y(t):=\norm{u_i^{q}(t)}_{L^2(\mathbb{R})}^2$. Combining (\ref{eq:L2q_eq1})--(\ref{eq:L2q_eq3}), we can obtain that for large $t$,
    $$
        \frac{\dd}{\dd t}Y(t)+\frac{(2q-1)d_i}{q}\norm{\partial_x(u_i^q)}_{L^2(\mathbb{R})}^2\leq C_3g(t),
    $$
    where 
    $$
    g(t):=\norm{W_q}_{L^2(|x|\leq ct)}^2+e^{-\kappa t}.
    $$
    Here $g\in L^1(T,\infty)$ by (\ref{eq:bound_2}). The one-dimensional Nash inequality yields
    $$
    \norm{u_i^q}_{L^2(\mathbb{R})}^3\leq C_4 \norm{u_i^q}_{L^1(\mathbb{R})}^2\norm{\partial_x (u_i^q)}_{L^2(\mathbb{R})},
    $$
    for some constant $C_4$. Then, by (\ref{eq:bound_3}), there exists some $T,C_5>0$ such that
    $$
    \frac{\dd}{\dd t}Y(t)+C_5 Y^3(t)\leq C_3 g(t),\quad t>T.
    $$
    It follows that
\begin{equation}\label{eq:z_ineq}
\frac{\dd}{\dd t}Z(t)\leq -C_5 Y^3(t)\leq 0, 
\end{equation}
with $Z(t):=Y(t)+C_3\int_t^\infty g(s)\dd s\geq 0$. Consequently, $Z$ has the same limit as $Y$ and this limit has to be zero, by (\ref{eq:z_ineq}). This proves $\norm{u_i(t)}_{L^{2q}(\mathbb{R})}\to 0$ as $t\to \infty$. Applying the argument to each slow component proves (\ref{eq:L2q extinction}).
\end{proof}
Heat-semigroup smoothing upgrades this $L^{2q}$ convergence to
uniform extinction. Henceforth, $\{e^{t\partial^2_x}\}_{t\geq 0}$ denotes the standard heat semigroup.
\begin{proposition}[$L^\infty$ extinction of slow species]\label{prop:Linfty extinction}
Under assumptions (\ref{eq:sec1_theorem2_condition1}) and (\ref{eq:sec1_theorem2_condition2}),
    $$
        \sum_{i=2}^N \norm{u_i(t)}_{L^\infty(\mathbb{R})}\to 0,\quad \mathrm{as} \ t\to\infty.
    $$
\end{proposition}
\begin{proof}
    Fix $i\geq2$. Since $W_q\leq 1$, we can obtain that
    $$
    \partial_t u_i\leq d_i \partial_x^2 u_i+r_iu_i.
    $$
    It follows from the parabolic comparison principle that
    $$
    u_i(t,x)\leq e^{r_i+d_i\partial_x^2}u_i(t-1,x),\quad t>1,\ x\in \mathbb{R}.
    $$
    By Young's convolution inequality and Proposition \ref{prop:L2q extinction}, 
    $$
    \norm{u_i(t)}_{L^\infty(\mathbb{R})}\leq e^{r_i}(4\pi d_i)^{-\frac{1}{4q}}\norm{u_i(t-1)}_{L^{2q}(\mathbb{R})}\to 0,\qquad as\ t\to\infty.
    $$
    This completes the proof.
\end{proof}
\subsection{Spreading and selection of the fastest species}
We are now in a position to conclude propagation properties of solutions to (\ref{eq:competition_diffusion}) with compactedly supported initial data.

Firstly, uniform extinction of the slow species makes their contribution to the equation for $u_1$ arbitrarily small at large times. A scalar Fisher--KPP comparison then determines the spreading speed and the limiting state behind the fronts.
\begin{proposition}[Selection and spreading of the fastest species]\label{prop:winnner selection}
Under assumptions (\ref{eq:sec1_theorem2_condition1}) and (\ref{eq:sec1_theorem2_condition2}),Under the same assumptions of Proposition \ref{prop:L2q extinction}, for every $\underline{c}<c_1<\overline{c}$,
    \begin{equation}\label{eq:selection}
        \norm{1-u_1(t)}_{L^\infty(|x|\leq \underline{c}t)}+\norm{u_1(t)}_{L^\infty(|x|\geq \overline{c}t)}\to 0,\quad \mathrm{as} \ t\to \infty.
    \end{equation}
\end{proposition}
\begin{proof}
    Fix $\underline{c}<c_1$. By Proposition \ref{prop:Linfty extinction}, for any small $\varepsilon$, there is $T>0$ such that
    $$
    \sum_{i=2}^N u_i(t,x)\leq\varepsilon,\quad \text{for all}\ t\geq T,\ x\in \mathbb{R}.
    $$
    Hence,
    $$
    \partial_t u_1\geq d_1\partial_x^2u_1+r_1u_1(1-\varepsilon-u_1^q).
    $$
    Choose $\varepsilon$ small enough that $c_\varepsilon=2\sqrt{(1-\varepsilon) d_1r_1}>\underline{c}$. At time $T$, choose a nonzero compactly supported datum $v_T$ satisfying $0\leq v_T\leq \min\{u_1(T,\cdot),(1-\varepsilon^{1/q})\}$. Let $v$ solve
    $$
    \partial_t =d_1\partial^2_xv+r_1v(1-\varepsilon-v^q),\quad v(T,\cdot)=v_T.
    $$
    Scalar comparison gives $u_1(t,x)\geq v(t,x)$ for $t\geq T$. The standard Fisher--KPP spreading theorem \cite{AronsonWeinberger1978} implies 
    $$
    (1-\varepsilon)^\frac{1}{q}\leq\liminf_{t\to \infty}\inf_{|x|\leq \underline{c}t}u_1(t,x)\leq 1
    $$
    Since $\varepsilon$ is arbitrary, $\lim_{t\to \infty}\norm{1-u_1(t)}_{L^\infty(|x|\leq \underline{c}t)}=0$. The remainder follows from Lemma \ref{lemma:gauss tail} with $\overline{c}>c_1$.
\end{proof}
\begin{remark}
    Together with uniform extinction, (\ref{eq:selection}) gives
    $$
    \lim_{t\to \infty}\sup_{|x|\leq ct}\left(|1-u_1(t,x)|+\sum_{i=2}^Nu_i(t,x)\right)=0,\quad \text{for every}\ 0\leq c<c_1.
    $$
    Thus, the solution to the Cauchy problem converges to $(1,0,\ldots)$ uniformly on every expanding interval strictly inside the spreading fronts.
\end{remark}

\noindent\textbf{\textit{Proof of Theorem \ref{thm:main_theorem2}.}}
The uniform extinction of the slow species follows from Proposition \ref{prop:Linfty extinction}, and the spreading properties of the fastest species follow from Proposition \ref{prop:winnner selection}.

For compactly supported initial data, the logarithmic Bramson correction of front locations can be obtained using level set arguments \cite{Bramson1983,Hamel2013}. Similar treatment for the two-species
critical competition system in \cite{AlfaroXiao2023} is still valid for $N$-species systems because of the same linearized structure at zero. 

The remaining assertion, namely extinction in $L^p(\mathbb{R})$ for every $1<p<\infty$, will be proved at the end of Section \ref{sec:bump1}. 
\section{Heat-type bump estimates}\label{sec:bump1}
In this section, we investigate the bump phenomenon for the N-species competition-diffusion system. As in the previous section, we assume that $u_i(t,x)$ ($i\geq 1$) are solutions to the Cauchy problem (\ref{eq:competition_diffusion}) with compactly supported initial data $u_i(0,\cdot)$ satisfying (\ref{eq:sec1_theorem2_condition1}). Moreover, we assume (\ref{eq:sec1_theorem2_condition2}) holds.

Henceforth, we denote $f^{-}(x):=\max\{0,-f(x)\}$ as the (nonnegative) negative part of a real-valued function $f(x)$. We first introduce an elementary heat-type lower comparison.
\begin{lemma}\label{lemma:heat lower comparison}
    Let $d,r>0$ and let $v\geq 0$ be a bounded solution of
    $$
    \partial_t v=d\partial^2_x v+rR(t,x)v
    $$
    on $(T,\infty)\times \mathbb{R}$, where $R$ is bounded. If $b(t):=\norm{R^-(t,\cdot)}_{L^\infty(\mathbb{R})}\in L^1(T,\infty)$, then,
    \begin{equation}\label{eq:heat lower bound}
    v(t,x)\geq \exp{(-r\int_T^t b(s)\dd s)}[e^{d(t-T)\partial^2_x}v(T,\cdot)](x),\quad \mathrm{for\ all}\ t\geq T,\ x\in \mathbb{R}.
    \end{equation}
\end{lemma}

The preceding section establishes uniform extinction of the slow species. We now show that their decay cannot be arbitrarily fast. More precisely, every nontrivial slow species admits a bump estimate of type (\ref{eq:heat lower bound}) with $\exp{(-r\int_T^t b(s)\dd s)}\ge C$. 

The key step is to prove that the $L^\infty$ norm of the negative part of $W_q$ is integrable in $t$. For $q>1$, thanks to the additional $\partial_x(u^{q/2}_i)$ term in the entropy dissipation, the desired time integrability follows directly from some extra gradient estimates associated with this term. On the other hand, for $q=1$, we overcome this difficulty by combining the $L^2$ extinction result with precise kernel estimates.
\subsection{Extra estimates for $q>1$}
Fix an intermediate speed $c_s<c<c_1$. When $q>1$, the entropy estimate (\ref{eq:bound_4}) yields 
\begin{equation}\label{eq:additional bound}
    \int_T^\infty\int_{-ct}^{ct}|\partial_x (u_i^{\frac{q}{2}})|^2\dd x\dd t<\infty,\quad i=1,2,\dots,N,
\end{equation}
for some large enough $T>0$. Throughout this subsection, $T$ may be increased many times if needed. We also use the boundness of (\ref{eq:bound_2}) and (\ref{eq:bound_3}) from Proposition \ref{prop:Linfty_L1} in the sequel. 

Firstly, as a direct application of (\ref{eq:additional bound}), we can prove the following.
\begin{lemma}\label{lemma:q larger than 1}
    Assume $q>1$. For every $c_s<c<c_1$ and $i\geq 2$
    \begin{equation}\label{eq:L2Linfty}
    \int_T^\infty\norm{u_i^q}^2_{L^\infty(|x|\leq ct)}\dd t<\infty,
    \end{equation}
    and
    \begin{equation}\label{eq:L2L1}
        \int_T^\infty \norm{u_i^qW_q}_{L^2(|x|\leq ct)}\dd t<\infty.
    \end{equation}
\end{lemma}
\begin{proof}
    Choose $c<\overline{c}<c_1$ and a truncation function $\chi_0\in C^\infty_c(\mathbb{R})$ such that $0\leq \chi_0\leq 1$ and 
    \begin{equation}\label{eq:cutoff}
        \chi_0(y)=1\quad \mathrm{for}\ |y|\leq c,\qquad \chi_0(y)=0\quad \mathrm{for}\ |y|\geq \overline{c}. 
    \end{equation}
    Set $\chi(t,x):=\chi_0(x/t)$. Its  derivatives satisfy
    \begin{equation}\label{eq:derivative chi}
    |\partial_x \chi|+|\partial_t\chi|\leq\frac{C}{t},\qquad |\partial_x^2\chi|\leq \frac{C}{t^2}.
\end{equation}
    It follows from the one-dimensional Sobolev inequality and (\ref{eq:bound_3}) that
    $$
    \begin{aligned}
        \norm{u_i^{q}}^2_{L^\infty(|x|\leq ct)}\leq \|\chi u_i^{\frac{q}{2}}\|_{L^\infty(\mathbb{R})}^4&\leq 4\|\chi u_i^{\frac{q}{2}}\|_{L^2(\mathbb{R})}^2\|\partial_x(\chi u_i^\frac{q}{2})\|_{L^2(\mathbb{R})}^2\\
        &\leq 4\left(\int_\mathbb{R}u_i^q(t,x)\dd x\right)\|\partial_x(\chi u_i^\frac{q}{2})\|_{L^2(\mathbb{R})}^2\\
        &\leq C_1\|\partial_x(\chi u_i^\frac{q}{2})\|_{L^2(\mathbb{R})}^2,
    \end{aligned}
    $$
    for some $C_1>0$ and large $t$. 
    
    Moreover, 
    $$
    \begin{aligned}
        \|\partial_x(\chi u_i^{\frac{q}{2}})\|^2_{L^2(\mathbb{R})}\dd x&\leq 2 \int_{-\overline{c}t}^{\overline{c}t}|\partial_x(u_i^{\frac{q}{2}})|^2\dd x+2\int_\mathbb{R}|\partial_x \chi|^2u_i^q(t,x)\dd x\\
        &\leq 2 \int_{-\overline{c}t}^{\overline{c}t}|\partial_x(u_i^{\frac{q}{2}})|^2\dd x+\frac{C_2}{t^2},
    \end{aligned}
    $$
    for some $C_2>0$. By (\ref{eq:additional bound}), the right-hand side is integrable in time. This proves (\ref{eq:L2Linfty}).
    
    Cauchy--Schwarz inequality yields  
    $$
    \begin{aligned}
    \int_T^\infty &\norm{u_i^qW_q}_{L^2(|x|\leq ct)}\dd t\\
    &\leq \left(\int_T^\infty\norm{u_i^q}^2_{L^\infty(|x|\leq ct)}\dd t\right)^\frac{1}{2}\left(\int_T^\infty\norm{W_q(t)}^2_{L^2(|x|\leq ct)}\dd t\right)^\frac{1}{2}<\infty,
    \end{aligned}
    $$
    where the second factor is finite by (\ref{eq:bound_2}). This completes the proof.
\end{proof}
\subsection{Duhamel's representation and integrability of $W_q^-$ for $q>1$}
A direct calculation yields 
$$
\begin{aligned}
    \partial_t(u_i^q)&=d_i \partial_x^2(u_i^q)-H_i+qr_iu_i^qW_q,
\end{aligned}
$$
where
$$
H_i:=q(q-1)d_iu_i^{q-2}(\partial_xu_i)^2=\frac{4(q-1)d_i}{q}|\partial_x (u_i^{\frac{q}{2}})|^2\geq 0.
$$
Thus, the equation of $W_q$ reads
$$
\begin{aligned}
    \partial_tW_q&=d_1\partial_x^2W_q+\sum_{i=2}^N(d_1-d_i)\partial_x^2(u_i^q)-qW_q\sum_{i=1}^Nr_iu_i^q+\sum_{i=1}^NH_i\\
    &=d_1\partial_x^2W_q-qr_1W_q+\sum_{i=2}^N(d_1-d_i)\partial_x^2(u_i^q)+F,
\end{aligned}
$$
with
$$F:=\sum_{i=1}^NH_i+qr_1W_q^2-qW_q\sum_{i=2}^N(r_i-r_1)u_i^q.$$

Fix $c_s<c<\overline{c}<c_1$. We define $\chi(t,x)=\chi_0(x/t)$ as in (\ref{eq:cutoff}).
A direct calculation yields that
\begin{equation}\label{eq:cutoff Sq}
    \begin{cases}
        \partial_t(\chi u_i^q)=d_i\partial_x^2(\chi u_i^q)+G_i\quad (i\geq 2),\\
        \partial_t(\chi W_q)=d_1\partial_x^2(\chi W_q)-qr_1(\chi W_q)+\sum_{i=2}^N(d_1-d_i)\partial_x^2(\chi u_i^q)+F^\ast,
    \end{cases}
\end{equation}
with
$$
G_i:=\chi(qr_iu_i^qW_q-H_i)+\hat{G}_i,\quad F^\ast=\chi F+\hat{F},
$$
where 
$$
\begin{aligned}
    \hat{G}_i&:=(\partial_t\chi)u_i^q-2d_i(\partial_x \chi)\partial_x(u_i^q)-d_i (\partial_x^2\chi)u_i^q,\\
    \hat{F}&:=(\partial_t\chi)W_q-2d_1(\partial_x\chi)(\partial_x W_q)-d_1(\partial_x^2\chi)W_q\\
    &\qquad -2\sum_{i=2}^N(d_1-d_i)(\partial_x\chi)[\partial_x(u_i^q)]-\sum_{i=2}^N (d_1-d_i)(\partial_x^2\chi)u_i^q.
\end{aligned}
$$
The integrability of sources in (\ref{eq:cutoff Sq}) can be concluded as follows.
\begin{lemma}\label{lemma:decomposition}
    Assume $q>1$. The source terms $G_i$ $\mathrm{(}i\geq 2\mathrm{)}$ and $F^\ast$ admit decompositions 
    $$
    \begin{aligned}
        G_i=G_i^{(1)}+G_i^{(2)}\ (i\geq 2),\quad 
        F^\ast=F^{(1)}+F^{(2)},
    \end{aligned}
    $$
    such that 
    \begin{equation}\label{eq:gf integrability}
    G_i^{(1)},F^{(1)}\in L^1((T,\infty);L^1(\mathbb{R}));\quad G_i^{(2)},F^{(2)}\in L^1((T,\infty);L^2(\mathbb{R})).
    \end{equation}
\end{lemma}
\begin{proof}
    We choose 
    $$
    \begin{aligned}
        G_i^{(1)}&=-\chi H_i +\hat{G}_i,\quad G_i^{(2)}=qr_i \chi u_i^qW_q,\\
        F^{(1)}&=\chi\left(\sum_{i=1}^N H_i+qr_1W_q^2\right)+\hat{F},\quad F^{(2)}=-q\chi W_q\sum_{i=2}^N(r_i-r_1)u_i^q.
    \end{aligned}
    $$
    It follows from (\ref{eq:bound_2}) and (\ref{eq:additional bound}) that 
    $$
    \int_T^\infty\int_{\mathbb{R}}\chi(H_i+W_q^2)\dd x\dd t\leq \int_T^\infty\int_{-\overline{c}t}^{\overline{c}t}(H_i+W_q^2)\dd x\dd t<\infty,\quad (i=1,2,\dots,N).
    $$
    Since the derivatives of $\chi$ are supported on $ct\leq |x|\leq \overline{c}t$, the exponential control of $u_i$ and $\partial_x u_i$ for $i\geq 2$ on this intermediate annulus(see Lemma \ref{lemma:exponential inter annuli}) and (\ref{eq:derivative chi}) gives
    \begin{equation}\label{eq:L1L1 estimate}
\norm{(\partial_x\chi)\partial_x (u_i^q)}_{L^1(\mathbb{R})}+\norm{(\partial_t\chi)u_i^q}_{L^1(\mathbb{R})}\leq Ce^{-\kappa t},\quad \norm{(\partial_x^2\chi)u_i^q}_{L^1(\mathbb{R})}\leq \frac{C}{t}e^{-\kappa t},
    \end{equation}
    for some $C,\kappa>0$; hence, these $L^1$ norms are time-integrable. Lemma \ref{lemma:exponential inter annuli} also implies that $W_q$ tends to zero exponentially fast for any $ct\leq |x|\leq \overline{c}t$ and sufficiently large $t$. Hence, (\ref{eq:L1L1 estimate}) still holds after replacing $u_i^q$ by $W_q$. The above statements prove $G_i^{(1)}$, $F^{(1)}\in L^1((T,\infty);L^1(\mathbb{R}))$.
    Moreover, (\ref{eq:L2L1}) gives $G_i^{(2)}$, $F^{(2)}\in L^1((T,\infty);L^2(\mathbb{R}))$. This completes the proof.
\end{proof}
Solving $\chi u_i^q$-equations in (\ref{eq:cutoff Sq}) by using Duhamel's formula and inserting the results into the $\chi W_q$-equation yields that
$$
\begin{aligned}
    \partial_t(\chi W_q)&=d_1\partial_x^2(\chi W_q)-qr_1(\chi W_q)+F^\ast\\
    &\qquad +\sum_{i=2}^N(d_1-d_i)\partial_x^2\left[e^{d_i(t-T)\partial_x^2}(\chi u_i^q)(T)+\int_T^t e^{d_i(t-s)\partial_x^2}G_i(s)\dd s\right].
\end{aligned}
$$
By applying Duhamel's formula and Fubini's theorem, we obtain that
$$
    \begin{aligned}
    \chi W_q&=E(t-T)[\chi W_q](T)+\int_T^tE(t-s)F^\ast(s)\dd s\\
&\qquad +\sum_{i=2}^N K_i(t-T)(\chi u_i^q)(T)+\sum_{i=2}^N\int_T^t K_i(t-s)G_i(s)\dd s,
\end{aligned}
$$
where 
$$
\begin{aligned}
    E(t)f&:=e^{-qr_1t}e^{d_1t\partial_x^2}f,\\
    K_i(t)f&:=(d_1-d_i)\int_0^t E(t-s)\partial_x^2(e^{^{d_is\partial_x^2}}f)\dd s\quad (i\geq 2),
\end{aligned}
$$
for any $f\in L^1(\mathbb{R})\cup L^2(\mathbb{R})$.
\begin{lemma}\label{lemma:ek}
    For $p=1,2$ and $t>0$, there exists $C>0$ such that
    \begin{equation}\label{eq:e}
        \norm{E(t)f}_{L^\infty(\mathbb{R})}\leq C e^{-qr_1t}t^{-\frac{1}{2p}}\norm{f}_{L^p(\mathbb{R})},
    \end{equation}
    and 
    \begin{equation}\label{eq:k}
        \norm{K_i(t)f}_{L^\infty(\mathbb{R})}\leq C\frac{t^{-\frac{1}{2p}}}{1+t}\norm{f}_{L^p(\mathbb{R})}.
    \end{equation}
    Consequently, the operator norms of both $E(t)$ and $K(t)$ are integrable in time.
\end{lemma}
\begin{proof}
Estimate (\ref{eq:e}) directly follows from elementary properties of the heat kernel. For $K_i$, the semigroup property yields
$$
K_i(t)f=(d_1-d_i)\int_0^t e^{-qr_1(t-s)}\partial_x^2
e^{[d_1(t-s)+d_is]\partial_x^2}f\dd s.
$$
Since 
$$
d_1(t-s)+d_i s\geq \min\{d_1,d_i\}t,\quad \text{for all}\ 0\leq s\leq t,
$$
we can obtain that
$$
\begin{aligned}
    \|K_i(t)f\|_{L^\infty(\mathbb{R})}&\leq Ct^{-1-\frac{1}{2p}}\left(\int_0^te^{-qr_1(t-s)}\,\dd s\right)\|f\|_{L^p(\mathbb{R})}\\
&\leq Ct^{-1-\frac{1}{2p}}\min\{t,1\}\|f\|_{L^p(\mathbb{R})}\\
&\leq C\frac{t^{-\frac{1}{2p}}}{1+t}\|f\|_{L^p(\mathbb{R})},
\end{aligned}
$$
for some $C>0$. Hence, (\ref{eq:k}) holds. The time-integrability of operator norms follows because $1/(2p)<1$.
\end{proof}
For $q>1$, by applying the decompositions of sources in Lemma \ref{lemma:decomposition} and the preceding estimates of kernel $E(t)$ and $K_i(t)$, we can obtain the $L^1$ time-integrability of the spatial $L^\infty$ norm of $W_q^-$.
\begin{proposition}\label{prop:integrability_q_large}
     For every $q>1$ and sufficiently large $T$,
     \begin{equation}\label{eq:integrable_q_large}
        \int_T^\infty\norm{\chi(t,\cdot)W_q(t,\cdot)}_{L^\infty(\mathbb{R})}\dd t<\infty.
    \end{equation}
    Moreover,
     \begin{equation}\label{eq:negative part_q large}
         \int_T^\infty\|W_q^-(t,\cdot)\|_{L^\infty(\mathbb{R})}\dd t<\infty.
     \end{equation}
\end{proposition}
\begin{proof}
At time $T$, it follows from Proposition \ref{prop:Linfty_L1} that $(\chi W_q)(T)\in L^2(\mathbb{R})$ and $\chi u_i^q(T)\in L^1(\mathbb{R})$ ($i\geq 2$). Therefore, Lemma \ref{lemma:ek} yields
$$
\int_T^\infty\norm{E(t-T)[\chi W_q](T)}_{L^\infty(\mathbb{R})}\dd t<\infty,\quad \int_T^\infty\norm{K_i(t-T)(\chi u_i^q)(T)}_{L^\infty(\mathbb{R})}\dd t<\infty.
$$
We decompose $F^\ast$ each $G_i$ ($i\geq 2$) according to Lemma \ref{lemma:decomposition}. Then, by Lemma \ref{lemma:ek} and Young's convolution inequality on the time half-line, we obtain that
$$
\begin{aligned}
    \int_T^\infty\norm{\int_T^tE(t-s)F^\ast(s)\dd s}_{L^\infty(\mathbb{R})}\dd t\leq C\int_T^\infty(\|F^{(1)}(t)\|_{L^1(\mathbb{R})}+\|F^{(2)}(t)\|_{L^2(\mathbb{R})})\dd t,\\
    \int_T^\infty\norm{\int_T^tK_i(t-s)G_i(s)\dd s}_{L^\infty(\mathbb{R})}\dd t\leq C\int_T^\infty(\|G_i^{(1)}(t)\|_{L^1(\mathbb{R})}+\|G_i^{(2)}(t)\|_{L^2(\mathbb{R})})\dd t,
\end{aligned}
$$
for every $t\geq T$. The right-hand sides are finite by (\ref{eq:gf integrability}). Consequently, (\ref{eq:integrable_q_large}) holds. Inside $|x|\leq ct$, we have $\chi=1$, and hence $W_q^-\leq |\chi W_q|$. Outside the expanding cone, the bound $u_1\leq 1$ gives $W_q^-\leq \sum_{i=2}^Nu_i^q.$
Since $c>c_s$, the exterior tail estimates of Lemma \ref{lemma:gauss tail} imply
    $$
    \|W_q^-\|_{L^\infty(\mathbb{R})}\leq \norm{\chi W_q}_{L^\infty(\mathbb{R})}+Ce^{-\kappa t},
    $$
    for large $t$. Integrating this inequality implies (\ref{eq:negative part_q large}).
\end{proof}    
\subsection{The critical case $q=1$}\label{sec:critical case}
Let the truncation function $\chi(t,x)=\chi_0(x/t)$ defined as (\ref{eq:cutoff}). When $q=1$, $W_1=1-\sum_{i=1}^N u_i$, and (\ref{eq:cutoff Sq}) is rewritten as 
$$
\begin{cases}
        \partial_t(\chi u_i)=d_i\partial_x^2(\chi u_i)+G_i\quad (i\geq 2),\\
        \partial_t(\chi W_1)=d_1\partial_x^2(\chi W_1)-r_1(\chi W_1)+\sum_{i=2}^N(d_1-d_i)\partial_x^2(\chi u_i)+F^\ast,
    \end{cases}
$$
where 
$$
G_i:=r_i\chi u_i W_1+\hat{G}_i,\quad F^\ast=r_1\chi W_1^2-\sum_{i=2}^N\left[(r_i-r_1)\chi u_i W_1\right]+\hat{F},
$$
and $\hat{F}$, $\hat{G}_i$ ($i\geq 2$) are defined as in the previous subsection with $q=1$.

It is clear that for $q=1$, $H_i$ itself vanishes, but the products $\chi u_i W_1$ remain in $G_i$, and in $F^\ast$ when the growth rates differ. If one can obtain that
\begin{equation}\label{eq:fail estimate}
\int_T^\infty\int_{-ct}^{ct}\abs{\partial_x\sqrt{u_i}}\dd x\dd t<\infty,
\end{equation}
for large enough $T$ and an intermediate speed $c_s<c<c_1$, the proof of Lemma \ref{lemma:q larger than 1} is applicable to derive estimates of $u_i W_1$ of type (\ref{eq:L2L1}). However, due to the loss of the $|\partial_x(u_i^{q/2})|^2$ term in the entropy dissipation when $q=1$, the extra gradient estimate (\ref{eq:fail estimate}) does not automatically hold.

We can overcome this obstacle by separating these “problematic” terms from the Duhamel's representation of $\chi W_1$ and applying the $L^2$ extinction results as established in Proposition \ref{prop:L2q extinction} of Section \ref{sec:selection}.
\begin{proposition}\label{prop:integrability_q_one}
(\ref{eq:integrable_q_large}) and (\ref{eq:negative part_q large}) hold for $q=1$ and all sufficiently large $T>0$.
\end{proposition}
\begin{proof}
    For $q=1$, the Duhamel's representation of $\chi W_1$ reads
    \begin{equation}\label{eq:S_1_initial}
        \begin{aligned}
        \chi W_1&=E(t-T)[\chi W_1](T)+\int_T^tE(t-s)F^\ast(s)\dd s\\
        &\quad +\sum_{i=2}^N\int_T^t K_i(t-s)G_i(s)\dd s +\sum_{i=2}^N K_i(t-T)\chi u_i(T).
    \end{aligned}
    \end{equation}
    By collecting all terms containing $\chi u_i W_1$, (\ref{eq:S_1_initial}) can be rewritten as 
    \begin{equation}\label{1}
        \chi W_1=P(t)+\sum_{i=2}^N\int_T^tB_i(t-s)[\chi u_i W_1](s)\dd s,
    \end{equation}
    where 
    $$
    B_i(t):=r_iK_i(t)-(r_i-r_1)E(t),
    $$
    and 
    $$
    \begin{aligned}
       P(t,x)&:=E(t-T)[\chi W_1](T)+\int_T^tE(t-s)[r_1\chi W_1^2+\hat{F}](s)\dd s\\
        &\qquad +\sum_{i=2}^NK_i(t-T)\chi u_i(T)+\sum_{i=2}^N\int_T^t K_i(t-s)\hat{G}_i(s)\dd s.
    \end{aligned}
    $$
    As shown in the proof of Lemma \ref{lemma:decomposition}, $\chi W_1^2,\hat{F},\hat{G}_i\in L^1((T,\infty);L^1(\mathbb{R}))$. It is easy to verify that $P\in L^1((T,\infty);L^\infty(\mathbb{R}))$. 

    Lemma \ref{lemma:ek} implies that for every $i\geq 2$, there exists $C_i>0$ such that
    \begin{equation}\label{2}
    \norm{B_i(t)f}_{L^\infty(\mathbb{R})}\leq C_i(\frac{t^{-\frac{1}{4}}}{1+t}+e^{-qr_1t}t^{-\frac{1}{4}})\norm{f}_{L^2(\mathbb{R})}
    \end{equation}
    For any $T>0$, we define 
    $$\varepsilon_T:=\max_{2\leq i\leq N}\sup_{t\geq T}\norm{u_i(t)}_{L^2(\mathbb{R})}.$$ 
    Proposition \ref{prop:L2q extinction} gives $\varepsilon_T\to 0$ as $T\to \infty$.
    For simplicity, we set
    $$
    B(t)=\sum_{i=2}^NC_i(\frac{t^{-\frac{1}{4}}}{1+t}+e^{-qr_1t}t^{-\frac{1}{4}})\in L^1(0,\infty),$$
    $$
    X(t)=\norm{\chi W_1}_{L^\infty(\mathbb{R})},\quad p(t)=\norm{P(t)}_{L^\infty(\mathbb{R})}\in L^1(T,\infty).
    $$
    Then, using \eqref{1} together with \eqref{2}, we obtain
    \begin{equation}\label{eq:x}
    X(t)\leq p(t)+\varepsilon_T\int_T^t B(t-s)X(s)\dd s,
    \end{equation}
    Fix $T_0>T$ and integrate (\ref{eq:x}) over $[T,T_0]$ yields
    $$
    \begin{aligned}
        \int_T^{T_0}X(t)\dd t&\leq \int_T^{T_0} p(t)\dd t+\varepsilon_T\int_T^{T_0}\int_T^t B(t-s)X(s)\dd s \dd t\\
        &=\int_T^{T_0} p(t)\dd t+\varepsilon_T\int_T^{T_0}X(s)\left(\int_s^{T_0} B(t-s)\dd t\right)\dd s\\
        &\leq \int_T^{T_0} p(t)\dd t+\varepsilon_T\norm{B}_{L^1(0,\infty)}\int_T^{T_0}X(s)\dd s.
    \end{aligned}
    $$
    We choose $T$ large enough that $\varepsilon_T\norm{B}_{L^1(0,\infty)}<1$. Then
    $$
    \int_T^{T_0}X(t)\dd t\leq \frac{1}{1-\varepsilon_T\norm{B}_{L^1(0,\infty)}}\int_T^{T_0} p(t)\dd t.
    $$
    By Letting $T_0\to \infty$, it is clear that (\ref{eq:integrable_q_large}) also hold when $q=1$. (\ref{eq:negative part_q large}) follows by the same argument as Proposition \ref{prop:integrability_q_large}.
\end{proof}
\subsection{Proof of the heat-type bumps}
In the previous subsections, we have shown that for any $q\geq 1$, $\|W_q^-(t,\cdot)\|_{L^\infty(\mathbb{R})}\in L^1(T,\infty)$ for large enough $T$. Hence, Lemma \ref{lemma:heat lower comparison} is applicable to provide heat-type lower bump estimates of every slow species $u_i$ ($i\geq 2$). 

\noindent \textbf{\textit{Proof of Theorem \ref{thm:main_theorem3}}.}
    Fix a nontrivial slow component $u_i$ ($i\geq 2$). We set $b(t)=\|W_q^-(t,\cdot)\|_{L^\infty}\in L^1(T,\infty)$ for some $T>0$ and set
    $$
    \beta_i=\exp\left(-r_i\int_T^\infty b(s)\dd s\right)>0.
    $$
    Applying Lemma \ref{lemma:heat lower comparison}  to the equation for $u_i$ gives 
    $$
    u_i(t,x)\geq \beta_i \left(e^{d_i(t-T)\partial^2_x}u_i(T,\cdot)\right)(x),\quad \text{for all}\ x\in \mathbb{R}.
    $$
    Choose $R_i>0$ and set $\alpha_i:=\int_{-R_i}^{R_i}u_i(T,x)\dd x>0$. It follows that
    $$
    \begin{aligned}
        u_i(t,x)&\geq \frac{\beta_i}{\sqrt{4\pi d_i(t-T)}}\int_{-R_i}^{R_i}\exp\left({-\frac{(x-y)^2}{4d_i(t-T)}}\right)u_i(T,y)\dd y\\
        &\geq \frac{\alpha_i\beta_i}{\sqrt{4\pi d_i(t-T)}}\exp\left(-\frac{(|x|+R_i)^2}{4d_i(t-T)}\right).
    \end{aligned}
    $$
    This proves (\ref{eq:sec1_main_thm_bump_1}). For $t\geq 2\max\{T,1\}$, and $|x|\leq D\sqrt{t}$, 
$$
\frac{(|x|+R_i)^2}{4d_i(t-T)}\leq\frac{(D+R_i)^2}{2d_i},\quad (t-T)^{-\frac{1}{2}}\geq t^{-\frac{1}{2}}.
$$ 
The Gaussian factor in (\ref{eq:sec1_main_thm_bump_1}) is therefore bounded below by a positive constant depending on $D$, proving (\ref{eq:sec1_main_thm_bump_2}).

\begin{remark}
Uniform extinction of the slow species does not imply extinction
in $L^1(\mathbb R)$. Indeed, for every nontrivial slow component,
(\ref{eq:sec1_main_thm_bump_2}) yields, for any fixed $D>0$,
$$
\|u_i(t)\|_{L^1(\mathbb{R})}
\geq
\int_{-D\sqrt t}^{D\sqrt t}u_i(t,x)\,\dd x
\geq2DC_{i,D},
$$
for all sufficiently large $t$. Consequently,
$$
\liminf_{t\to\infty}\norm{u_i(t)}_{L^1(\mathbb R)}>0.
$$
\end{remark}

\noindent\textbf{\textit{Completion of the proof of
Theorem \ref{thm:main_theorem2}.}}
Here we are in a position to prove the extinction of nontrivial slow species in $L^p(\mathbb{R})$ for every $1<p<\infty$. 

Fix $c_s<c<\overline c<c_1$ and choose $\chi$ as in
(\ref{eq:cutoff}). We claim that for large $T$,
\begin{equation}\label{eq:bound_L1}
    \sup_{t\geq T}\norm{u_i(t)}_{L^1(\mathbb{R})}<\infty.
\end{equation}
Indeed, Lemma \ref{lemma:gauss tail} ensures that $u_i(t,\cdot)\in L^1(\mathbb{R})$ for every $i\geq 2$ and $t>0$. Set $M_i(t)=\norm{u_i(t,\cdot)}_{L^1(\mathbb{R})}$. A direct calculation yields
$$
\begin{aligned}
    \frac{\dd}{\dd t}M_i(t)&=r_i\int_{\mathbb{R}}u_i(t,x)W_q(t,x)\dd x\\
    &\leq r_i \int_{\mathbb{R}}u_i(t,x)(\chi W_q)(t,x)\dd x+r_i \int_{|x|>ct}u_i(t,x)\dd x\\
    &\leq r_i \|\chi W_q\|_{L^\infty(\mathbb{R})}M_i(t)+r_i \int_{|x|>ct}u_i(t,x)\dd x.
\end{aligned}
$$
Since $c>c_s$, Lemma \ref{lemma:gauss tail} implies
$$
\frac{\dd}{\dd t}M_i(t)\leq r_i \|\chi W_q\|_{L^\infty(\mathbb{R})}M_i(t)+Ce^{-\kappa t},\quad t\geq T,
$$
for some $\kappa,C>0$. By Propositions \ref{prop:integrability_q_large} and \ref{prop:integrability_q_one}, $\|\chi W_q\|_{L^\infty(\mathbb{R})}\in L^1(T,\infty)$ for every $q\geq1$ and sufficiently large $T$. Hence,
$$
M_i(t)
\leq
\exp\left(r_i\int_T^\infty \|(\chi W_q)(s)\|_{L^\infty(\mathbb{R})}\dd s\right)\left(
M_i(T)+C\int_T^\infty e^{-\kappa s}\dd s
\right),
\quad t\geq T.
$$
In particular, (\ref{eq:bound_L1}) holds.

For every $1<p<\infty$, interpolation yields 
$$
\norm{u_i(t)}^p_{L^p(\mathbb{R})}\leq \norm{u_i(t)}^{p-1}_{L^\infty(\mathbb{R})}\norm{u_i(t)}_{L^1(\mathbb{R})}\to 0,\quad \text{as}\ t\to \infty,
$$
by Proposition \ref{prop:Linfty extinction}. Since there are only finitely many slow species,
$$
\lim_{t\to\infty}
\sum_{i=2}^N \left(\norm{u_i(t)}_{L^p(\mathbb{R})}+\norm{u_i(t)}_{L^\infty(\mathbb{R})}\right)=0,
\quad \text{for all}\ 1<p<\infty.
$$
This completes the proof.



\begin{thebibliography}{99}

\bibitem{AlfaroXiao2023}
M.~Alfaro and D.~Xiao,
\newblock Lotka--Volterra competition-diffusion system: the critical competition case,
\newblock \emph{Communications in Partial Differential Equations}
\textbf{48} (2023), no.~2, 182--208.

\bibitem{AronsonWeinberger1978}
D.~G.~Aronson and H.~F.~Weinberger,
\newblock Multidimensional nonlinear diffusion arising in population genetics,
\newblock \emph{Advances in Mathematics}
\textbf{30} (1978), no.~1, 33--76.

\bibitem{Bramson1983}
M.~Bramson,
\newblock Convergence of solutions of the Kolmogorov equation to travelling waves,
\newblock \emph{Memoirs of the American Mathematical Society}
\textbf{44} (1983), no.~285.

\bibitem{Hamel2013}
F.~Hamel, J.~Nolen, J.-M.~Roquejoffre, and L.~Ryzhik,
\newblock A short proof of the logarithmic Bramson correction in Fisher-KPP equations,
\newblock \emph{Networks and Heterogeneous Media}
\textbf{8} (2013), no.~1, 275--289.

\bibitem{Fisher1937}
R.~A.~Fisher,
\newblock The wave of advance of advantageous genes,
\newblock \emph{Annals of Eugenics}
\textbf{7} (1937), no.~4, 355--369.

\bibitem{KPP1937}
A.~N.~Kolmogorov, I.~G.~Petrovskii, and N.~S.~Piskunov,
\newblock A study of the equation of diffusion with increase in the quantity
of matter, and its application to a biological problem,
\newblock \emph{Bulletin of Moscow State University, Series A: Mathematics and Mechanics}
\textbf{1} (1937), 1--25.

\bibitem{CMX2025}
J.~Cai, S.~Ma, and D.~Xiao,
\newblock Lotka--Volterra competition system with critical competition case
in high-dimensional space,
\newblock \emph{Discrete and Continuous Dynamical Systems -- B}
\textbf{30} (2025), no.~9, 3505--3525.

\bibitem{Goh1977}
B.~S.~Goh,
\newblock Global stability in many-species systems,
\newblock \emph{The American Naturalist}
\textbf{111} (1977), no.~977, 135--143.

\bibitem{CW2025}
C.-H.~Chang and C.-H.~Wu,
\newblock Bistable wavefronts of three-species competition-diffusion systems: Weak competition strategy between two species, \newblock \emph{Journal of Differential Equations}
\textbf{422} (2025), 529--561.

\bibitem{JW2026}
J.-S.~Jiang and C.-H.~Wu,
\newblock Spreading speeds for a diffusive three-species system with strong
competition,
\newblock \emph{Journal of Dynamics and Differential Equations}
\textbf{38} (2026), no.~2, 865--894.

\bibitem{GG2022}
J.-S.~Guo and K.~Guo,
\newblock Traveling waves for a three-species competition system with two weak aboriginal competitors,
\newblock \emph{Differential and Integral Equations}
\textbf{35} (2022), nos.~11--12, 819--832.

\bibitem{CH2016}
C.-C.~Chen and L.-C.~Hung,
\newblock Nonexistence of traveling wave solutions, exact and semi-exact traveling wave solutions for diffusive Lotka--Volterra systems of three
competing species,
\newblock \emph{Communications on Pure and Applied Analysis}
\textbf{15} (2016), no.~4, 1451--1469.

\bibitem{CHMU2012}
C.-C.~Chen, L.-C.~Hung, M.~Mimura, and D.~Ueyama,
\newblock Exact travelling wave solutions of three-species competition--diffusion systems, \newblock \emph{Discrete and Continuous Dynamical Systems -- B}
\textbf{17} (2012), no.~8, 2653--2669.

\bibitem{LLM2010}
G.~Lin, W.-T.~Li, and M.~Ma,
\newblock Traveling wave solutions in delayed reaction diffusion systems with applications to multi-species models,
\newblock \emph{Discrete and Continuous Dynamical Systems--B}
\textbf{13} (2010), no.~2, 393--414.

\bibitem{CL2021}
R.~S.~Cantrell and K.-Y.~Lam,
\newblock On the evolution of slow dispersal in multispecies communities,
\newblock \emph{SIAM Journal on Mathematical Analysis}
\textbf{53} (2021), no.~4, 4933--4964.

\bibitem{TF1980}
M.~M.~Tang and P.~C.~Fife,
\newblock Propagating fronts for competing species equations with diffusion,
\newblock \emph{Archive for Rational Mechanics and Analysis}
\textbf{73} (1980), 69--77.

\bibitem{LLL2020}
Q.~Liu, S.~Liu, and K.-Y.~Lam,
\newblock Asymptotic spreading of interacting species with multiple fronts I: A geometric optics approach,
\newblock \emph{Discrete and Continuous Dynamical Systems}
\textbf{40} (2020), no.~6, 3683--3714.

\bibitem{KanOn1997}
Y.~Kan-on,
\newblock Fisher wave fronts for the Lotka--Volterra competition model
with diffusion,
\newblock \emph{Nonlinear Analysis: Theory, Methods and Applications}
\textbf{28} (1997), no.~1, 145--164.

\bibitem{GL2019}
L.~Girardin and K.-Y.~Lam,
\newblock Invasion of open space by two competitors: spreading properties of monostable two-species competition-diffusion systems,
\newblock \emph{Proceedings of the London Mathematical Society}
\textbf{119} (2019), no.~5, 1279--1335.

\bibitem{KanOn1995}
Y.~Kan-on,
\newblock Parameter dependence of propagation speed of travelling waves for competition-diffusion equations,
\newblock \emph{SIAM Journal on Mathematical Analysis}
\textbf{26} (1995), no.~2, 340--363.

\bibitem{Gardner1982}
R.~A.~Gardner,
\newblock Existence and stability of travelling wave solutions of competition models: A degree theoretic approach,
\newblock \emph{Journal of Differential Equations}
\textbf{44} (1982), no.~3, 343--364.

\bibitem{Girardin2019}
L.~Girardin,
\newblock The effect of random dispersal on competitive exclusion---A review,
\newblock \emph{Mathematical Biosciences}
\textbf{318} (2019), 108271.

\bibitem{Hastings1983}
A.~Hastings,
\newblock Can spatial variation alone lead to selection for dispersal?,
\newblock \emph{Theoretical Population Biology}
\textbf{24} (1983), no.~3, 244--251.


\bibitem{DockeryEtAl1998}
J.~Dockery, V.~Hutson, K.~Mischaikow, and M.~Pernarowski,
\newblock The evolution of slow dispersal rates: a reaction diffusion model,
\newblock \emph{Journal of Mathematical Biology}
\textbf{37} (1998), no.~1, 61--83.
\end{thebibliography}
\end{document}